\documentclass[10pt]{amsart}
\usepackage{amssymb}

\newtheorem{thm}{Theorem}[section]
\newtheorem{lemma}[thm]{Lemma}
\newtheorem{proposition}[thm]{Proposition}

\newtheorem{definition}[thm]{Definition}
\newtheorem{corollary}[thm]{Corollary}

\newcommand{\p}{\mathbb{P}}

\newcommand{\cf}{\mathrm{cf}}
\newcommand{\cof}{\mathrm{cof}}

\begin{document}

\title{Coherent Adequate Forcing and Preserving CH}

\author{John Krueger and Miguel Angel Mota}

\address{John Krueger \\ Department of Mathematics \\ 
University of North Texas \\
1155 Union Circle \#311430 \\
Denton, TX 76203}
\email{jkrueger@unt.edu}

\address{Miguel Angel Mota \\ Department of Mathematics \\ 
University of Toronto \\
Toronto, Ontario \\
Canada M5S 2E4}
\email{motagaytan@gmail.com}

\date{June, 2014}

\thanks{2010 \emph{Mathematics Subject Classification.} 
Primary: 03E40. Secondary: 03E05}

\thanks{\emph{Key words and phrases.} Forcing, side conditions, adequate sets, 
coherent adequate sets}

\begin{abstract}
We develop a general framework for forcing with 
coherent adequate sets on $H(\lambda)$ as side conditions, 
where $\lambda \ge \omega_2$ is a cardinal of uncountable cofinality. 
We describe a class of forcing posets which we call 
coherent adequate type forcings. 
The main theorem of the paper is that any 
coherent adequate type forcing preserves CH. 
We show that there exists a forcing poset for adding a club subset 
of $\omega_2$ with finite conditions while preserving CH, solving a 
problem of Friedman \cite{friedman}.
\end{abstract}

\maketitle

The method of side conditions, invented by 
{Todor\v cevi\' c} (\cite{todor}), describes a style of forcing in which 
elementary substructures are included in the conditions 
of a forcing poset to ensure that the forcing poset preserves cardinals. 
Friedman (\cite{friedman}) and Mitchell (\cite{mitchell}) independently 
took the first steps in generalizing the method from adding generic 
objects of size $\omega_1$ to adding larger objects 
by defining forcing posets with finite conditions for adding a 
club subset of $\omega_2$. 
Neeman (\cite{neeman}) was the first to simplify the side conditions of 
Friedman and Mitchell and present a generally applicable technique 
for forcing on $\omega_2$ with finite conditions.

Krueger (\cite{jk21}) developed an alternative framework for 
forcing objects of size $\omega_2$ with finite conditions, using 
adequate sets of models as side conditions. 
An adequate set of models consists of countable models which are 
pairwise membership comparable up to some initial segment. 
Later Krueger (\cite{jk23}) introduced the idea of coherent adequate sets, 
which requires the existence of isomorphisms between certain models in an adequate set. 
This idea combined adequate sets with an isomorphism structure 
originally used by {Todor\v cevi\' c} \cite{todor} in the 
context of forcing on $\omega_1$. 
Coherent adequate sets were applied in \cite{jk23} to define a 
strongly proper forcing poset which forces $\Box_{\omega_1}$.

The present paper makes advances on the framework of coherent adequate sets. 
We present a more general development in the context of 
$H(\lambda)$ for a cardinal $\lambda \ge \omega_2$ of uncountable cofinality, rather 
than just $H(\omega_2)$ as was treated in \cite{jk23}. 
We define a class of forcing posets which we call 
coherent adequate type forcings. 
The main theorem of the paper is that any coherent adequate type 
forcing preserves CH. 
More generally, any coherent adequate type forcing on $H(\lambda)$, 
where $2^\omega < \lambda$ is a cardinal of uncountable cofinality, 
collapses $2^\omega$ to have size $\omega_1$ and forces CH. 
We describe coherent adequate type forcings 
for adding a square sequence and 
for adding a club to a fat stationary subset of $\omega_2$. 

The forcing posets of Friedman, Mitchell, and Neeman for adding a club subset 
of $\omega_2$ with finite conditions all force that $2^\omega = \omega_2$. 
Any forcing poset which has strongly generic conditions for 
countable models will add reals, including those defined in this paper. 
These earlier forcings for adding clubs with finite conditions 
can be factored in many ways so that the quotient forcing also 
has strongly generic conditions in the intermediate extension. 
For this reason, these posets add $\omega_2$ many distinct reals. 
Friedman (\cite{friedman}) asked whether it is possible to add a club 
subset of $\omega_2$ with finite conditions while preserving CH. 
We solve this problem by defining a forcing poset which adds a club 
to a fat stationary set and falls in the class of coherent adequate 
type forcings.

Finally we show that, under CH, the forcing poset consisting of 
finite coherent adequate subsets of $H(\lambda)$ ordered by inclusion, 
where $\lambda \ge \omega_2$ 
is regular, is $\omega_2$-c.c.\ and therefore preserves all cardinals. 

\bigskip

Section 1 develops the basic ideas of adequate and coherent 
adequate sets. 
This development is almost self contained, except for three results 
for which we refer the reader to the corresponding results of \cite{jk21} 
for proofs. 
Differences between the current paper and earlier papers on adequate sets 
include the consideration of a more general 
context, namely countable elementary substructures of $H(\lambda)$ for some 
$\lambda \ge \omega_2$ of uncountable cofinality, rather 
than just $H(\omega_2)$. 
Also we omit the assumption 
that $2^{\omega_1} = \omega_2$. 
Since we are not interested in taking initial segments of models as was done 
in \cite{jk21}, 
the set $\Lambda$ which is used to compare models can 
be taken to be $\omega_2 \cap \cof(\omega_1)$. 
Some simplifications of the material in \cite{jk21} follow from these new conventions 
and from the presence of isomorphisms between models.

Section 2 proves the main result of the paper, that any coherent adequate 
type forcing preserves CH. 
Section 3 proves that if CH fails, then any coherent adequate type forcing 
on $H(\lambda)$, where $\lambda > 2^\omega$ is a cardinal of uncountable cofinality, 
collapses $2^\omega$ to have size $\omega_1$, preserves 
$(2^\omega)^+$, and forces CH. 
Sections 4 and 5 develop the technical machinery for amalgamating 
conditions over countable elementary substructures. 
Sections 6 and 7 give examples of coherent adequate type forcings. 
In Section 6 we review the poset from \cite{jk23} for adding a square 
sequence. 
In Section 7 we define a coherent adequate type forcing poset 
which adds a club to a fat stationary subset of $\omega_2$. 

Section 8 presents general results for amalgamating coherent adequate 
sets over elementary substructures of size $\omega_1$. 
These results are not needed for the present paper, but could be useful for 
future applications. 
We show that the forcing poset consisting of finite coherent adequate 
subsets of $H(\lambda)$ ordered by end-extension is $\omega_2$-c.c.

\bigskip

The general development of coherent adequate sets presented in 
Sections 1, 4, 5, and 8 is due to Krueger.

Asper\'o and Mota (\cite{mota}) 
proved recently that for any cardinal $\lambda \ge \omega_2$ of 
uncountable cofinality, 
the forcing poset consisting of finite symmetric systems of countable 
elementary substructures of $H(\lambda)$ ordered by inclusion preserves CH. 
A symmetric system is similar to a coherent adequate set, except that it 
does not have the adequate structure. 
Also {Todor\v cevi\' c} pointed out to the authors 
that in an unpublished result from the 1980s he proved using a 
different argument that forcing with 
finite sets of countable elementary substructures of $H(\omega_2)$ 
with isomorphisms, of the 
type described at the end of \cite{todor}, preserves CH and adds an 
$\omega_1$-Kurepa tree. 

Neither of these results show how to force with side conditions together 
with another finite set of objects to preserve CH, nor do they imply 
anything regarding adequate set forcing. 
By arguments of Miyamoto \cite{miyamoto}, any coherent adequate type forcing  
on $H(\lambda)$ adds an $\omega_1$-tree with $\lambda$ 
many cofinal branches, for any regular cardinal $\lambda \ge \omega_2$. 
Thus $\omega_1$-Kurepa trees exist in coherent adequate type forcing 
extensions.

\section{Coherent Adequate Sets}

In this section we present the basic framework of coherent adequate sets. 
We will assume throughout the paper that 
$\lambda \ge \omega_2$ is a fixed cardinal of uncountable cofinality. 
This implies that any countable subset of $H(\lambda)$ is a member of $H(\lambda)$.
We also fix a predicate $Y \subseteq H(\lambda)$, which we assume codes a 
well-ordering of $H(\lambda)$ among other things.

Let $\mathcal X$ denote the set of $N \subseteq H(\lambda)$ such that $N$ is countable 
and $N \prec (H(\lambda),\in,Y)$. 
We introduce a way to compare members of $\mathcal X$. 
Fix $\Lambda$ a cofinal subset of $\omega_2 \cap \cof(\omega_1)$.

\begin{definition}
For $M \in \mathcal X$, let $\Lambda_M$ denote the set of 
$\beta \in \Lambda$ such that 
$$
\beta = \min(\Lambda \setminus (\sup(M \cap \beta)).
$$
\end{definition}

Since $M$ is countable, it has countably many limit points. 
It follows easily that $\Lambda_M$ is countable.

\begin{lemma}
Let $M \in \mathcal X$ and $\beta \in \Lambda_M$. 
If $\beta_0 \in \Lambda \cap \beta$ then $M \cap [\beta_0,\beta) \ne \emptyset$. 
\end{lemma}

\begin{proof}
Suppose for a contradiction that $M \cap [\beta_0,\beta) = \emptyset$. 
Then $\sup(M \cap \beta) \le \beta_0$ . 
So $\beta = \min(\Lambda \setminus (\sup(M \cap \beta))) \le \beta_0$, which 
contradicts that $\beta_0 < \beta$.
\end{proof}

\begin{lemma}
For $M$ and $N$ in $\mathcal X$, $\Lambda_M \cap \Lambda_N$ has a largest element.
\end{lemma}

We omit the proof and refer the reader to Lemma 2.4 of \cite{jk21}, whose proof is 
nearly identical to that needed in the present context.

\begin{definition}
For $M$ and $N$ in $\mathcal X$, let $\beta_{M,N}$ denote the largest element of 
$\Lambda_M \cap \Lambda_N$. 
The ordinal $\beta_{M,N}$ is called the \emph{comparison point} of $M$ and $N$.
\end{definition}

Given a set $K$ in $\mathcal X$, let $K'$ denote the set 
$(K \cap \omega_2) \cup (\lim(K \cap \omega_2))$.
 
\begin{lemma}
Let $M$ and $N$ be in $\mathcal X$. 
Then $M' \cap N' \subseteq \beta_{M,N}$.
\end{lemma}

The proof is almost the same as the proof of Proposition 2.6 of \cite{jk21}, 
so we skip it.

If $K$ and $M$ are in $\mathcal X$ and $K \subseteq M$, then an easy argument shows that 
$\Lambda_K \subseteq \Lambda_M$. 
It follows that for all $N$ in $\mathcal X$, 
$\max(\Lambda_K \cap \Lambda_N) \le \max(\Lambda_M \cap \Lambda_N)$. 
So $\beta_{K,N} \le \beta_{M,N}$.

We define relations $<$, $\le$, and $\sim$ on $\mathcal X$. 
Let $M < N$ if $M \cap \beta_{M,N} \in N$. 
Let $M \sim N$ if $M \cap \beta_{M,N} = N \cap \beta_{M,N}$. 
Let $M \le N$ if either $M < N$ or $M \sim N$.

\begin{definition}
A subset $A$ of $\mathcal X$ is said to be \emph{adequate} 
if for all $M$ and $N$ in $A$, either $M < N$, $M \sim N$, or $N < M$.
\end{definition}

Note that any subset of an adequate set is adequate. 
Also if $A$ is finite and adequate, $M \in \mathcal X$, and $A \in M$, then 
$A \cup \{ M \}$ is adequate.

Suppose that $M < N$. 
Then $M \cap \beta_{M,N} \in N$. 
But $\beta_{M,N} \in \Lambda_M$ implies that 
$\beta_{M,N} = \min(\Lambda \setminus (\sup(M \cap \beta_{M,N})))$. 
Hence $\beta_{M,N}$ is definable in $H(\lambda)$ from 
$M \cap \beta_{M,N}$. 
Since $N$ is elementary in $H(\lambda)$, 
$\beta_{M,N}$ is in $N$.

\begin{lemma}
Let $M$ and $N$ be in $\mathcal X$ and $\beta \in \Lambda$.
\begin{enumerate}
\item If $M \cap \omega_2 \subseteq \beta$, then $\beta_{M,N} \le \beta$.
\item If $\beta < \beta_{M,N}$ and $\{ M, N \}$ is adequate, then 
$M \cap N \cap [\beta,\beta_{M,N}) \ne \emptyset$.
\end{enumerate}
\end{lemma}

\begin{proof}
(1) Suppose for a contradiction that $\beta < \beta_{M,N}$. 
Then since $\beta \in \Lambda$ and $\beta_{M,N} \in \Lambda_M$, 
$M \cap [\beta,\beta_{M,N}) \ne \emptyset$ by Lemma 1.2. 
This contradicts that $M \cap \omega_2 \subseteq \beta$.

(2) Without loss of generality assume that $M \le N$. 
Then $M \cap \beta_{M,N} \subseteq N$. 
By Lemma 1.2, fix $\xi$ in $M \cap [\beta,\beta_{M,N})$. 
Then $\xi$ is in $M \cap \beta_{M,N}$ and hence in $N$. 
So $\xi$ is in $M \cap N \cap [\beta,\beta_{M,N})$.
\end{proof}

Next we define remainder points, which describe the overlap of models past their 
comparison point.

\begin{definition}
Let $M$ and $N$ be in $\mathcal X$ and assume that $\{ M, N \}$ is adequate. 
Define $R_M(N)$ as the set of $\beta$ satisfying either:
\begin{enumerate}
\item $N \le M$ and $\beta = \min(N \setminus \beta_{M,N})$, or
\item there is $\gamma \in M \setminus \beta_{M,N}$ such that 
$\beta = \min(N \setminus \gamma)$.
\end{enumerate}
\end{definition}

The set $R_M(N)$ is called 
the set of \emph{remainder points of $N$ over $M$}. 
This set is always finite, since otherwise there would be a common 
limit point of $M$ and $N$ 
greater than $\beta_{M,N}$, contradicting Lemma 1.5. 
For a more detailed proof, see Proposition 2.9 of \cite{jk21}.

Suppose that $M < N$. 
Then by definition, if $\zeta \in R_M(N)$ then there is $\gamma \in M \setminus \beta_{M,N}$ 
such that $\zeta = \min(N \setminus \gamma)$. 
On the other hand, consider $\zeta \in R_N(M)$. 
Since $\beta_{M,N}$ is in $N$ as noted above, 
if $\zeta = \min(M \setminus \beta_{M,N})$, then again there is 
$\gamma \in N \setminus \beta_{M,N}$ such that 
$\zeta = \min(M \setminus \gamma)$, namely $\gamma = \beta_{M,N}$. 
So remainder points of models $M$ and $N$ are given just 
by condition (2) in Definition 1.8 in the case that $M < N$ or $N < M$. 
Condition (1) is only 
relevant when $M \cap \beta_{M,N} = N \cap \beta_{M,N}$.

Given an adequate set $A$, define $R_A$ by letting 
$$
R_A  = \bigcup \{ R_N(M) : M, N \in A \}.
$$

\begin{definition}
For a given set $S \subseteq \omega_2$, a set $A \subseteq \mathcal X$ is 
\emph{$(S)$ adequate} if it is adequate and $R_A \subseteq S$.
\end{definition}

For the rest of the paper we let $\Lambda := \omega_2 \cap \cof(\omega_1)$. 
Note that $\Lambda$ is a definable subset of $H(\lambda)$.

\bigskip

Now we move on to coherent adequate sets. 
We will consider isomorphisms between models in $\mathcal X$. 

Let $M$ and $N$ be in $\mathcal X$ and let $\sigma : M \to N$. 
We say that $\sigma$ is an \emph{isomorphism} if $\sigma$ is 
a bijection and for all $a$ and $b$ in $M$,
\begin{itemize}
\item $a \in b$ iff $\sigma(a) \in \sigma(b)$;
\item $a \in Y$ iff $\sigma(a) \in Y$.
\end{itemize}
In other words, $\sigma$ is an isomorphism if it is an isomorphism 
in the usual model theoretic sense between the structures 
$(M,\in,Y \cap M)$ and $(N,\in,Y \cap N)$. 
We say that $M$ and $N$ are \emph{isomorphic} 
if there exists an isomorphism from $M$ to $N$. 
Note that the isomorphism relation is an equivalence relation.

For a model $M$ in $\mathcal X$, the elementarity of $M$ 
in $H(\lambda)$ 
implies that 
$M$ satisfies the axiom of extensionality. 
It follows that $(M,\in,Y \cap M)$ is isomorphic to a unique transitive structure 
$(\overline M,\in,\overline Y)$ 
by a unique isomorphism $\sigma_M : M \to \overline M$ given by 
the recursive equation $\sigma_M(a) = \sigma_M[a \cap M]$. 
Note that if $a \in M$ is countable, then $a \subseteq M$, so 
$\sigma_M(a) = \sigma_M[a]$.

By the uniqueness of the transitive collapse, 
a standard argument shows that $M$ and $N$ in $\mathcal X$ are isomorphic iff 
the structures $(M,\in,Y \cap M)$ and 
$(N,\in,Y \cap N)$ have the same transitive collapse. 
In that case, $\sigma_N^{-1} \circ \sigma_M$ is an isomorphism 
from $M$ to $N$, which we denote by $\sigma_{M,N}$. 
Also the uniqueness of the transitive collapsing map easily implies that 
$\sigma_{M,N}$ is the unique isomorphism from $M$ to $N$. 
Note that this map satisfies that if $a \in M$ is countable, 
then $\sigma_{M,N}(a) = \sigma_{M,N}[a]$. 
Also by the uniqueness of isomorphisms, 
if $M$, $N$, and $P$ are isomorphic, 
then $\sigma_{M,P} = \sigma_{N,P} \circ \sigma_{M,N}$.

\begin{lemma}
Let $M$, $N$ and $K$ be in $\mathcal X$ such that $M$ and $N$ 
are isomorphic and $K \in M$. 
Let $L := \sigma_{M,N}(K)$. 
Then:
\begin{enumerate}
\item $L \in \mathcal X$;
\item $K$ and $L$ are isomorphic and $\sigma_{K,L} = \sigma_{M,N} \restriction K$.
\end{enumerate}
\end{lemma}

\begin{proof}
(1) As $L \subseteq N$ and $(N,\in,N \cap Y) \prec (H(\lambda),\in,Y)$, 
it suffices to show that $\mathfrak L := (L,\in,L \cap Y) \prec (N,\in,N \cap Y)$. 
Let $\mathfrak K := (K,\in,M \cap Y)$. 
By the elementarity of $M$ and $N$, the predicates $K \cap Y$ are $L \cap Y$ are 
in $M$ and $N$ respectively. 
Let $b_1, \ldots, b_k$ be in $L$ and let $\varphi(x_1,\ldots,x_k)$ be a formula in 
the language of the structure $(H(\lambda),\in,Y)$. 
Let $a_i := \sigma_{N,M}(b_i)$ for $i = 1, \ldots, k$. 
Then $(L,\in,L \cap Y) \models \varphi[b_1,\ldots,b_k]$ iff 
$N \models {\ulcorner\varphi\urcorner}^{\mathfrak L}[b_1,\ldots,b_k]$ iff 
$M \models {\ulcorner\varphi\urcorner}^{\mathfrak K}[a_1,\ldots,a_k]$ iff 
$M \models \varphi[a_1,\ldots,a_k]$ iff 
$N \models \varphi[b_1,\ldots,b_k]$, where the third equivalence follows from 
the fact that $\mathfrak K$ is an elementary substructure of $(M,\in,M \cap Y)$.

(2) Since $K$ is countable, $\sigma_{M,N}(K) = \sigma_{M,N}[K]$. 
So $\sigma_{M,N} \restriction K$ is a bijection from $K$ to $L$. 
It is obvious that $\sigma_{M,N} \restriction K$ preserves the 
predicates $\in$ and $Y$ since $\sigma_{M,N}$ does. 
So $\sigma_{M,N} \restriction K$ is an isomorphism of $K$ to $L$. 
By the uniqueness of isomorphisms, $\sigma_{M,N} \restriction K = \sigma_{K,L}$.
\end{proof}

\begin{lemma}
Suppose that $M$ and $N$ are in $\mathcal X$ and are isomorphic. 
Let $\sigma := \sigma_{M,N}$. 
Assume that $K$ and $L$ are in $M \cap \mathcal X$. 
Then:
\begin{enumerate}
\item $\sigma(\beta_{K,L}) = \beta_{\sigma(K),\sigma(L)}$;
\item $K < L$ iff $\sigma(K) < \sigma(L)$;
\item $K \sim L$ iff $\sigma(K) \sim \sigma(L)$;
\item $K \le L$ iff $\sigma(K) \le \sigma(L)$;
\item $\sigma(R_K(L)) = R_{\sigma(K)}(\sigma(L))$;
\item if $\{ K, L \}$ is adequate then $\{ \sigma(K), \sigma(L) \}$ 
is adequate.
\end{enumerate}
\end{lemma}

\begin{proof}
The lemma follows from the fact for any $P$ and $Q$ in $\mathcal X$, 
the objects and relations 
$\beta_{P,Q}$, $P < Q$, $P \le Q$, and $R_P(Q)$ are definable in 
$H(\lambda)$ from $P$ and $Q$. 
\end{proof}

We now introduce an additional requirement on isomorphisms. 
Let us say that $M$ and $N$ in $\mathcal X$ are \emph{strongly isomorphic} 
if they are isomorphic and for all $a \in M \cap N$, 
$\sigma_{M,N}(a) = a$. 
We write $M \cong N$ to indicate that $M$ and $N$ are strongly isomorphic.

\begin{lemma}
Let $M$, $N$ and $K$ be in $\mathcal X$ such that $M$ and $N$ 
are strongly isomorphic and $K \in M$. 
Let $L := \sigma_{M,N}(K)$. 
Then $K$ and $L$ are strongly isomorphic.
\end{lemma}

\begin{proof}
By Lemma 1.10, $K$ and $L$ are isomorphic and $\sigma_{K,L} = \sigma_{M,N} \restriction K$. 
Let $a \in K \cap L$ be given. 
Then $a \in M \cap N$. 
So $\sigma_{K,L}(a) = \sigma_{M,N}(a) = a$.
\end{proof}

\begin{lemma}
Suppose that $N$, $N'$, and $N^*$ are in $\mathcal X$ and are strongly isomorphic. 
Let $a \in N' \cap N^*$. 
Then $\sigma_{N',N}(a) = \sigma_{N^*,N}(a)$.
\end{lemma}

\begin{proof}
Since $a \in N' \cap N^*$, $\sigma_{N',N^*}(a) = a$. 
So $\sigma_{N',N}(a) = \sigma_{N^*,N}(\sigma_{N',N^*}(a)) = 
\sigma_{N^*,N}(a)$.
\end{proof}

\begin{lemma}
Let $N$, $N'$, $N^*$, $K$, and $L$ be in $\mathcal X$ such that 
$N$, $N'$, and $N^*$ are strongly isomorphic, $K$ and $L$ are strongly isomorphic, 
and $K$ and $L$ are in $N$. 
Let $M := \sigma_{N,N'}(K)$ and $P := \sigma_{N,N^*}(L)$. 
Then $M$ and $P$ are strongly isomorphic.
\end{lemma}

\begin{proof}
By Lemma 1.10, $K$ and $M$ are isomorphic and $L$ and $P$ are isomorphic. 
Since $K$ and $L$ are isomorphic and being isomorphic is an equivalence relation, 
$M$ and $P$ are isomorphic. 
To show that $M$ and $P$ are strongly isomorphic, 
let $a \in M \cap P$ and we will show that $\sigma_{M,P}(a) = a$. 
Let $b := \sigma_{N',N}(a)$. 
Since $a \in M \cap P$, $a \in N' \cap N^*$. 
By Lemma 1.13, $\sigma_{N^*,N}(a) = \sigma_{N',N}(a) = b$. 
Now $a \in M$ implies that $b \in K$, and 
$a \in P$ implies that $b \in L$. 
So $b \in K \cap L$. 
Since $K$ and $L$ are strongly isomorphic, $\sigma_{K,L}(b) = b$. 
So $\sigma_{M,P}(a) = \sigma_{L,P}(\sigma_{K,L}(\sigma_{M,K}(a))) = 
\sigma_{N,N^*}(\sigma_{K,L}(\sigma_{N',N}(a))) = 
\sigma_{N,N^*}(\sigma_{K,L}(b)) = \sigma_{N,N^*}(b) = a$. 
\end{proof}

\begin{definition}
Let $A$ be a subset of $\mathcal X$. 
We say that $A$ is \emph{coherent adequate} if $A$ is adequate, and 
for all $M$ and $N$ in $A$:
\begin{enumerate}
\item if $M \sim N$, then $M$ and $N$ 
are strongly isomorphic;
\item if $M < N$, then there exists $N' \in A$ isomorphic to $N$ 
such that $M \in N'$;
\item if $M$ and $N$ are isomorphic and $K \in M \cap A$, then 
$\sigma_{M,N}(K) \in A$.
\end{enumerate}
\end{definition}

Let $M$ and $N$ be in $\mathcal X$. 
Since $\omega_1 \le \beta_{M,N}$, $M \sim N$ 
implies that $M \cap \omega_1 = N \cap \omega_1$. 
Also if $M < N$, then $M \cap \omega_1$ is an initial segment of $M \cap \beta_{M,N}$ 
and hence is in $N$. 
It follows that if $\{ M, N \}$ is adequate, then 
$M \sim N$ iff $M \cap \omega_1 = N \cap \omega_1$. 
Also if $M$ and $N$ are isomorphic, then since $\omega_1$ is definable in 
$H(\lambda)$, $\sigma_{M,N}(\omega_1) = \omega_1$. 
It easily follows that $M \cap \omega_1 = N \cap \omega_1$.

As a consequence of these observations, if $A$ is coherent adequate and $M$ and 
$N$ are in $A$, then the following are equivalent:
\begin{itemize}
\item $M \sim N$; 
\item $M \cap \omega_1 = N \cap \omega_1$; 
\item $M$ and $N$ are isomorphic;
\item $M$ and $N$ are strongly isomorphic.
\end{itemize}
Moreover, $M < N$ iff $M \cap \omega_1 < N \cap \omega_1$.

Given a set $S \subseteq \omega_2$, a set 
$A \subseteq \mathcal X$ is \emph{$(S)$ coherent adequate} if 
$A$ is coherent adequate and $R_A \subseteq S$.

\begin{lemma}
Let $A$ be a finite coherent adequate set. 
\begin{enumerate}
\item If $M \in \mathcal X$ and $A \in M$, then $A \cup \{ M \}$ is coherent adequate.
\item $M \cap A$ is coherent adequate.
\end{enumerate}
\end{lemma}

\begin{proof}
(1) We already observed that $A \cup \{ M \}$ is adequate, and the coherent properties 
are immediate. 
(2) Requirements (1) and (3) in the definition of coherent adequate are obvious. 
For requirement (2), suppose that $K$ and $N$ are in $A \cap M$ 
and $K < N$. 
Since $A$ is coherent, there is $N' \in A$ isomorphic to $N$ such that $K \in N'$. 
If $N' \in M$ then we are done, so assume not. 
As $N' \cap \omega_1 = N \cap \omega_1 < M \cap \omega_1$, 
$N' < M$. 
So there is $M' \in A$ isomorphic to $M$ with $N' \in M'$. 
Let $N^* = \sigma_{M',M}(N')$, which is in $M \cap A$. 
Since $K \in N'$ and $N' \in M'$, $K \in M'$. 
But also $K \in M$, so $\sigma_{M',M}(K) = K$. 
Since $K \in N'$, $K = \sigma_{M',M}(K) \in \sigma_{M,M'}(N') = N^*$. 
So $K \in N^*$, $N^*$ is isomorphic to $N$, and 
$N^* \in M \cap A$.
\end{proof}

\begin{lemma}
Let $A$ be a finite coherent adequate set. 
Suppose that $\{ M_0, \ldots, M_k \}$ is adequate and consists 
of strongly isomorphic sets. 
Assume that $A \in M_0 \cap \cdots \cap M_k$. 
Then $A \cup \{ M_0, \ldots, M_k \}$ is coherent adequate.
\end{lemma}

\begin{proof}
It is obvious that $A \cup \{ M_0, \ldots, M_k \}$ is adequate and satisfies 
requirements (1) and (2) in the definition of coherent. 
Requirement (3) follows from the fact that for any $i, j \le k$, if 
$K \in M_i \cap A$, then $K \in M_j$, and therefore 
$\sigma_{M_i,M_j}(K) = K \in A$.
\end{proof}

\begin{lemma}
Let $M$ and $N$ be in $\mathcal X$ and assume that $M$ and $N$ are 
strongly isomorphic. 
Then $M \cap P_{\omega_1}(M \cap N) \subseteq N$. 
In particular, if $\{ M, N \}$ is adequate and $M$ and $N$ are 
strongly isomorphic, then $M \cap P_{\omega_1}(\beta_{M,N}) \subseteq N$.
\end{lemma}

\begin{proof}
Let $a \in M \cap P_{\omega_1}(M \cap N)$ be given. 
Since $a$ is countable and $\sigma_{M,N}$ is the identity on $M \cap N$, 
$\sigma_{M,N}(a) = \sigma_{M,N}[a] = a$. 
So $a \in N$. 
If in addition we know that $\{ M , N \}$ is adequate, then 
$M \sim N$. 
So if $a \in M \cap P_{\omega_1}(\beta_{M,N})$, then $a$ is a countable 
subset of $M \cap \beta_{M,N} = N \cap \beta_{M,N}$. 
Hence $a \in M \cap P_{\omega_1}(M \cap N)$ and therefore $a \in N$.
\end{proof}

\begin{lemma}
Let $A$ be a coherent adequate set. 
Suppose that $M$ and $N$ are in $A$ and $M \le N$. 
Then $M \cap P_{\omega_1}(\beta_{M,N}) \subseteq N$.
\end{lemma}

\begin{proof}
If $M$ and $N$ are strongly isomorphic, then we are done 
by the preceding lemma. 
Assume that $M < N$. 
Since $A$ is coherent, fix $N' \in A$ isomorphic to $N$ such that $M \in N'$. 
Then $N$ and $N'$ are strongly isomorphic. 
Consider $a \in M \cap P_{\omega_1}(\beta_{M,N})$. 
Then $a \in N' \cap P_{\omega_1}(\beta_{M,N})$. 
Since $M \in N'$, $\beta_{M,N} \le \beta_{N,N'}$. 
So $a \in N' \cap P_{\omega_1}(\beta_{N',N})$. 
As $N$ and $N'$ are strongly isomorphic, $a \in N$ by the preceding lemma.
\end{proof}

\section{Preserving CH}

Fix for the remainder of this section a set $S \subseteq \omega_2$ 
such that $S \cap \cof(\omega_1)$ is stationary and 
a set $\mathcal Y \subseteq \mathcal X$ which is stationary 
in $P_{\omega_1}(H(\lambda))$. 
Also assume that $\mathcal Y$ is closed under isomorphisms, which means 
that whenever $M$ and $N$ are in $\mathcal Y$ and are isomorphic, 
and $K \in M \cap \mathcal Y$, then $\sigma_{M,N}(K) \in \mathcal Y$. 
Note that by Lemma 1.10, the set $\mathcal X$ itself is closed under 
isomorphisms.

A forcing poset $\p$ is said to be an \emph{$(S,\mathcal Y)$ 
coherent adequate type forcing} if there exists a natural number $m$ 
such that $\p$ consists of conditions of the form 
$(x_0,\ldots,x_m,A)$ satisfying:
\begin{enumerate}
\item[(I)] $x_0, \ldots, x_m$ are finite subsets of $H(\lambda)$;
\item[(II)] $A$ is a finite $(S)$ coherent adequate subset of $\mathcal Y$;
\item[(III)] if $(y_0,\ldots,y_m,B) \le (x_0,\ldots,x_m,A)$, 
$N$ and $N'$ are isomorphic sets in $B$, and 
$(x_0,\ldots,x_m,A) \in N$, then $\sigma_{N,N'}((x_0,\ldots,x_m,A))$ is in $\p$, and 
$(y_0,\ldots,y_m,B) \le \sigma_{N,N'}((x_0,\ldots,x_m,A))$;
\item[(IV)] if $M_0, \ldots, M_n$ are isomorphic sets in $\mathcal Y$ such that 
$\{ M_0, \ldots, M_n \}$ is $(S)$ coherent adequate and 
$(x_0,\ldots,x_m,A) \in M_0 \cap \cdots \cap M_n$, then there is a condition 
$(y_0,\ldots,y_m,B) \le (x_0,\ldots,x_m,A)$ such that 
$M_0, \ldots, M_n \in B$;
\item[(V)] for all $M \in A$, $(x_0,\ldots,x_m,A)$ is strongly $M$-generic.
\end{enumerate}
Regarding (V), recall that a condition $q$ is strongly $M$-generic if for any set $D$ which is 
dense in the forcing poset $M \cap \p$, $D$ is predense in $\p$ below $q$. 

We say that $\p$ is a \emph{coherent adequate type forcing} if 
$\p$ is an $(\omega_2,\mathcal X)$ coherent type forcing. 
Define $(S)$ coherent adequate and $(\mathcal Y)$ coherent adequate type 
forcings similarly. 
We interpret the above definition to include the possibility that the sequence $x_0,\ldots,x_m$ in 
a condition has length $0$, in which case conditions are just 
$(S,\mathcal Y)$ coherent adequate sets.

\begin{lemma}
Let $\p$ be an $(S,\mathcal Y)$ coherent adequate forcing. 
Then $\p$ preserves $\omega_1$.
\end{lemma}

\begin{proof}
Let $\dot f$ be a $\p$-name for a function from $\omega$ to $\omega_1$ and let $p$ be a condition. 
We will find $q \le p$ which forces that the range of $\dot f$ is bounded. 
Fix $\chi \geq \lambda$ regular such that $\dot f \in H(\chi)$. 
Since $\mathcal Y$ is stationary, we can find $N^*$ a countable elementary substructure of $H(\chi)$ 
with $\p$, $\dot f$, and $p$ in $N^*$ and $N := N^* \cap H(\lambda)$ in $\mathcal Y$.

By property (IV), there is $q = (y_0,\ldots,y_m,B) \le p$ such that $N \in B$. 
By property (V), $q$ is strongly $N$-generic. 
It is easy to see that for all $n < \omega$, the set of $s \in N \cap \p$ which decide the 
value of $\dot f(n)$ is dense in $N \cap \p$. 
It follows that for all $n < \omega$, if $r \le q$ decides the value of $\dot f(n)$, then 
that value is a member of $N$. 
So $q$ forces that $\dot f$ is bounded.
\end{proof}

We will prove that any $(S,\mathcal Y)$ coherent adequate type forcing 
preserves $\omega_2$ and CH. 
The proof will use the following technical result. 

\begin{lemma}
(CH) Suppose that $R_0, \ldots, R_k$ are subsets of $H(\lambda)$. 
Then for any set $z \in H(\lambda)$, there are $M$ and $N$ in $\mathcal Y$ satisfying 
the following:
\begin{enumerate}
\item $z \in M \cap N$;
\item $\{ M, N \}$ is $(S)$ coherent adequate; 
\item the structures $(M,\in,Y \cap M,R_0 \cap M,\ldots,R_k \cap M)$ 
and $(N,\in,Y \cap N,R_0 \cap N,\ldots,R_k \cap N)$ are 
elementary in $(H(\lambda),\in,Y,R_0,\ldots,R_k)$ and are isomorphic;
\item $M \cap \omega_2 \subseteq \min(N \setminus \beta_{M,N})$ and 
$\sigma_{M,N}(\min(M \setminus \beta_{M,N})) = 
\min(N \setminus \beta_{M,N})$.
\end{enumerate}
\end{lemma}

\begin{proof}
Since $S \cap \cof(\omega_1)$ is stationary and CH holds, we can fix $N^*$ satisfying:
\begin{enumerate}
\item $N^* \prec H(2^\lambda)$;
\item $|N^*| = \omega_1$;
\item $Y$, $z$, $S$, $\mathcal Y$, $R_0, \ldots, R_k$ are in $N^*$;
\item $\beta^* := N^* \cap \omega_2 \in S$;
\item $(N^*)^\omega \subseteq N^*$;
\end{enumerate}

As $\mathcal Y$ is stationary, we can fix $N$ in $\mathcal Y$ 
such that $z, \beta^* \in N$ and 
$$
(N,\in,Y \cap N,R_0 \cap N,\ldots,R_k \cap N) \prec (H(\lambda),\in,Y,R_0,\ldots,R_k).
$$

Let $\mathfrak N$ denote the structure 
$(N,\in,Y \cap N,R_0 \cap N,\ldots,R_k \cap N)$, 
let $\overline{\mathfrak N}$ denote its transitive collapse, 
and as usual let $\sigma_N$ denote the transitive collapsing map. 
Let $T$ be the relation defined by letting $T(a,b)$ hold if 
$a \in N \cap N^*$ and $\sigma_N(a) = b$. 
By CH, $H(\omega_1)$ has size $\omega_1$. 
Since $(N^*)^\omega \subseteq N^*$, $H(\omega_1)$ is a subset of $N^*$. 
It follows that $\overline{\mathfrak N}$ is a member and a subset of $N^*$. 
Also $N^*$ contains the sets 
$N \cap N^*$ and $N \cap \beta^*$ as members. 
It follows that the relation $T$ is a subset of $N^*$, and since it is 
countable it is a member of $N^*$.

The objects $N$ and $\beta^*$ witness the statement, satisfied by 
the structure $H(2^\lambda)$, that there 
exist $M$ and $\beta$ satisfying:
\begin{enumerate}
\item[(a)] $M$ is in $\mathcal Y$;
\item[(b)] $\beta \in S \cap \cof(\omega_1)$;
\item[(c)] $z$ and $\beta$ are in $M$;
\item[(d)] $N \cap \beta^* = M \cap \beta$;
\item[(e)] $N \cap N^* \subseteq M$;
\item[(f)] $(M,\in,Y,R_0 \cap M,\ldots,R_k \cap M) \prec 
(H(\lambda),\in,Y,R_0,\ldots,R_k)$;
\item[(g)] the transitive collapse of the structure 
$(M,\in,Y \cap M,R_0 \cap M,\ldots,R_k \cap M)$ is equal to 
$\overline{\mathfrak{N}}$;
\item[(h)] $T(a,b)$ implies that $\sigma_{M}(a) = b$.
\end{enumerate}
Since the parameters which appear in the above statement are all members 
of $N^*$, by elementarity we can fix $M$ and $\beta$ in $N^*$ satisfying the 
same statement.

Let us prove that $M$ and $N$ are as desired. 
We know that $M$ and $N$ are in $\mathcal Y$ and $z \in M \cap N$. 
Also the structures 
$(M,\in,Y \cap M,R_0 \cap M,\ldots,R_k \cap M)$ and 
$(N,\in,Y \cap N,R_0 \cap N,\ldots,R_k \cap N)$ are elementary in 
$(H(\lambda),\in,Y,R_0,\ldots,R_k)$. 
And since they have the same transitive collapse, they are isomorphic. 

We claim that $M$ and $N$ are strongly isomorphic. 
So let $a \in M \cap N$ be given, and we show that $\sigma_{N,M}(a) = a$. 
Since $M \in N^*$, it follows that $a \in N \cap N^*$. 
Let $b := \sigma_N(a)$. 
Then $T(a,b)$ holds. 
By (h), $\sigma_M(a) = b$. 
It follows that $\sigma_{N,M}(a) = \sigma_M^{-1}(\sigma_N(a)) = 
\sigma_M^{-1}(b) = a$.

We claim that $\beta_{M,N} \le \beta$. 
Since $M \in N^*$, $\Lambda_M \subseteq N^*$. 
But $\beta_{M,N} \in \Lambda_M$, so $\beta_{M,N} \in N^* \cap \omega_2 = 
\beta^*$. 
This shows that $\beta_{M,N} \le \beta^*$. 
Let $\xi := \sup(N \cap \beta_{M,N})$. 
As $\beta_{M,N} \le \beta^*$, $\xi$ is a limit point of 
$N \cap \beta^*$. 
Since $\beta_{M,N} \in \Lambda_N$, 
$\beta_{M,N} = \min(\Lambda \setminus \xi)$. 
Now $N \cap \beta^* = M \cap \beta$ implies that 
$N \cap \beta^* \subseteq \beta$. 
Since $\xi$ is a limit point of $N \cap \beta^*$, $\xi < \beta$. 
As $\beta_{M,N} = \min(\Lambda \setminus \xi)$ and 
$\beta \in \Lambda$, $\beta_{M,N} \le \beta$.

Now $M \cap \beta = N \cap \beta^*$ and $\beta_{M,N} \le \beta$ imply that 
\begin{enumerate}
\item[(i)] $M \sim N$;
\item[(ii)] $\min(M \setminus \beta_{M,N}) \ge \beta$;
\item[(iii)] $\min(N \setminus \beta_{M,N}) \ge \beta^*$. 
\end{enumerate}
Statement (i) is immediate. 
For (ii), if $\tau \in M \setminus \beta_{M,N}$, then $\tau \in M \setminus N$. 
As $M \cap \beta \subseteq N$, $\tau \ge \beta$. 
The proof of (iii) is similar.

By (i), $\{ M, N \}$ is adequate. 
And since $M$ and $N$ are strongly isomorphic, $\{ M, N \}$ is 
coherent adequate. 
As $\beta \in M$ and $\beta^* \in N$, 
(ii) and (iii) imply that $\min(M \setminus \beta_{M,N}) = \beta$ and 
$\min(N \setminus \beta_{M,N}) = \beta^*$. 
Also $M \in N^*$ implies that $M \cap \omega_2 \subseteq \beta^*$. 
This easily implies that $R_M(N) = \{ \beta^* \}$ and $R_N(M) = \{ \beta \}$. 
As $\beta$ and $\beta^*$ are in $S$, $\{ M, N \}$ is $(S)$ coherent adequate.

Since $\beta$ and $\beta^*$ are the first elements of $M$ and $N$ above 
their common intersection, clearly $\sigma_{M,N}(\beta) = \beta^*$. 
Hence $\sigma_{M,N}(\min(M \setminus \beta_{M,N})) = 
\min(N \setminus \beta_{M,N})$. 
Finally, $M \cap \omega_2 \subseteq \beta^* = 
\min(N \setminus \beta_{M,N})$.
\end{proof}

\begin{proposition}
(CH) Let $\p$ be an $(S,\mathcal Y)$ coherent adequate type forcing. 
If $p$ is a condition which forces that 
$\langle \dot f_i : i < \omega_2^V \rangle$ is a sequence of functions 
from $\omega$ to $\omega$, 
then there is $q \le p$ and $i < j$ such that 
$q$ forces that $\dot f_i = \dot f_j$.
\end{proposition}

\begin{proof}
Define a relation $R$ on $H(\lambda)$ 
by letting $R(z, i, n, m)$ if $z \in \p$ and 
$z \Vdash_\p \dot f_i(n) = m$. 
By Lemma 2.2, we can fix $M$ and $N$ in $\mathcal Y$ satisfying:
\begin{enumerate}
\item $p \in M \cap N$;
\item $\{ M, N \}$ is $(S)$ coherent adequate; 
\item the structures $(M,\in,Y \cap M,R \cap M)$ and 
$(N,\in,Y \cap N,R \cap N)$ are 
elementary in $(H(\lambda),\in,Y,R)$ and are isomorphic;
\item $M \cap \omega_2 \subseteq \min(N \setminus \beta_{M,N})$ and 
$\sigma_{M,N}(\min(M \setminus \beta_{M,N})) = 
\min(N \setminus \beta_{M,N})$.
\end{enumerate}

Let $\sigma := \sigma_{M,N}$. 
By the uniqueness of isomorphisms, $\sigma$ is an isomorphism between 
the structures described in (3) above. 
Let $i := \min(M \setminus \beta_{M,N})$ and 
$j := \min(N \setminus \beta_{M,N})$. 
Then $i < j$ and $\sigma(i) = j$.

We claim that for all $z \in M \cap \p$ and integers $n$ and $m$, if 
$z \Vdash_\p \dot f_i(n) = m$ then $\sigma(z) \Vdash_\p \dot f_j(n) = m$. 
For assume that $z \Vdash_\p \dot f_i(n) = m$. 
Then $R(z,i,n,m)$ holds. 
Since $z$, $i$, $n$, and $m$ are in $M$, $\sigma(i) = j$, 
and $\sigma$ is an isomorphism, 
$R(\sigma(z),j,n,m)$ holds. 
By definition of $R$, this means that $\sigma(z) \Vdash_\p \dot f_j(n) = m$.

Since $M$ and $N$ are isomorphic sets in $\mathcal Y$, 
$\{ M, N \}$ is $(S)$ coherent adequate, and 
$p \in M \cap N$, 
by property (IV) in the definition of an $(S,\mathcal Y)$ 
coherent adequate type forcing, 
there is a condition $q = (y,B)$ below $p$ such that $M, N \in B$. 
By property (V), $q$ is strongly $M$-generic.

We claim that $q$ forces that $\dot f_i = \dot f_j$, 
which completes the proof. 
So assume for a contradiction that there exists $r \le q$ and 
$n < \omega$ such that $r \Vdash_\p \dot f_i(n) \ne \dot f_j(n)$.

Let $D$ be the set of conditions $w$ in $\p \cap M$ such that 
for some $m$, $w \Vdash_\p \dot f_i(n) = m$. 
We claim that $D$ is dense in $\p \cap M$. 
So let $v \in \p \cap M$ be given. 
Clearly there exists $w$ and $m$ such that $w \le v$ and 
$w \Vdash_\p \dot f_i(n) = m$, and hence $R(w,i,n,m)$ holds.  
Since $v$, $i$, and $n$ are in $M$ and 
$M$ is elementary in $(H(\lambda),\in,Y,R)$, there is $w \leq v$ 
in $\p \cap M$ and $m$ such that $R(w,i,n,m)$ holds. 
Then $w$ is an extension of $v$ in $D$.

As $q$ is strongly $M$-generic, $D$ is predense below $q$. 
In particular, there is $w \in D$ such that $w$ is compatible with $r$. 
Fix $s \le w, r$. 
So $s \le w$, $M$ and $N$ are isomorphic models in $A_s$, and $w \in M$. 
By Property (III) in the definition of an $(S,\mathcal Y)$ 
coherent adequate type forcing, 
$s \le \sigma(w)$. 
Since $w \Vdash_\p \dot f_i(n) = m$, 
$\sigma(w) \Vdash_\p \dot f_j(n) = m$ by the claim above. 
As $s$ extends both $w$ and $\sigma(w)$, 
$s$ forces that $\dot f_i(n) = \dot f_j(n)$. 
But this contradicts that $s \le r$, since 
$r$ forces that $\dot f_i(n) \ne \dot f_j(n)$.
\end{proof}

\begin{corollary}
(CH) Let $\p$ be an $(S,\mathcal Y)$ 
coherent adequate type forcing. 
Then $\p$ preserves $\omega_2$ and $CH$.
\end{corollary}

\begin{proof}
The statement that $\p$ preserves $CH$ is immediate from the proposition. 
Suppose for a contradiction that 
$\p$ does not preserve $\omega_2$, and let $p$ be a condition 
which forces that $|\omega_2^V| = \omega_1$. 
Then we can find a sequence of names which $p$ forces is 
an enumeration of $\omega_1$ many distinct functions 
from $\omega$ to $\omega$ in order type $\omega_2^V$, contradicting 
the proposition.
\end{proof}

\section{Collapsing the Continuum}

In this section we analyze what happens when we force with a coherent adequate type 
forcing over a model in which CH fails. 
We will prove that in this context, the cardinal $(2^\omega)^V$ will be collapsed to have 
size $\omega_1$, its successor in $V$ will become $\omega_2$, and CH will hold in the extension . 

Let $\lambda$ be any cardinal with uncountable cofinality 
such that $2^\omega < \lambda$. 
Let $\langle r_i : i < 2^{\omega} \rangle$ be an enumeration of 
the power set of $\omega$ such that $r_i \ne r_j$ for all $i < j < 2^\omega$. 
Moreover, assume that $\langle r_i : i < 2^\omega \rangle$ is the first such enumeration 
according to the well-ordering of $H(\lambda)$ coded by the predicate $Y$. 
It follows that $Y$ codes the relation $Z$, where $Z(i,n)$ holds if 
$i < 2^\omega$ and $n \in r_i$. 

\begin{lemma}
Suppose that $M$ and $N$ are in $\mathcal X$ and are isomorphic. 
Then $\sigma_{M,N}(\alpha) = \alpha$ for all $\alpha \in M \cap 2^\omega$. 
Therefore $M \cap 2^\omega = N \cap 2^\omega$.
\end{lemma}

\begin{proof}
Let $\alpha \in M \cap 2^\omega$ be given. 
Since $r_i \ne r_j$ for all $i < j < 2^\omega$, it suffices 
to show that $r_{\alpha} = r_{\sigma_{M,N}(\alpha)}$. 
So let $n < \omega$ be given. 
Then $n \in r_\alpha$ iff $M \models Z(\alpha,n)$ iff $N \models Z(\sigma_{M,N}(\alpha),n)$ iff 
$n \in r_{\sigma_{M,N}(\alpha)}$.
\end{proof}

\begin{thm}
Let $\p$ be an $(S,\mathcal Y)$ 
coherent adequate type forcing. 
Let $\mu$ be the cardinal $2^\omega$. 
Then $\p$ collapses $\mu$ to have size $\omega_1$.
\end{thm}

\begin{proof}
Suppose that $p = (x_0,\ldots,x_k,A)$ is a condition, $M$ and $N$ are in $A$, 
and $M < N$. 
We claim that $M \cap \mu \in N$. 
By coherence, fix $N'$ in $A$ such that $N$ and $N'$ are isomorphic and 
$M \in N'$. 
Then $M \cap \mu \in N'$. 
By Lemma 3.1, $\sigma_{N',N}(M \cap \mu) = \sigma_{N',N}[M \cap \mu] = M \cap \mu$. 
Hence $M \cap \mu$ is in $N$.

Let $G$ be a generic filter for $\p$. 
Let $J := \{ N : \exists (x_0,\ldots,x_k,A) \in G \ (N \in A) \}$. 
Then for all $M$ and $N$ in $J$, if $M \cap \omega_1 = N \cap \omega_1$ 
then $M$ and $N$ are isomorphic. 
For in that case the conditions in $G$ witnessing and $M$ and $N$ are in $J$ have a 
common extension in $G$, and the set of models in this condition is coherent adequate. 
By Lemma 3.1, $M \cap \mu = N \cap \mu$. 
Similarly, if $M$ and $N$ are in $J$ and $M \cap \omega_1 < N \cap \omega_1$, then 
we can find a condition $(x_0,\ldots,x_k,A)$ in $G$ such that $M$ and $N$ are in $A$. 
By the previous paragraph, $M \cap \mu \in N$. 
In particular, $M \cap \mu \subseteq N \cap \mu$.

It follows that $K := \{ N \cap \mu : N \in J \}$ is well ordered by the subset relation in 
order type $\omega_1$. 
Using property (IV) in the definition of a coherent adequate type forcing, 
an easy density argument 
shows that for any $i < \mu$, there is $N$ in $J$ such that $i \in N \cap \mu$. 
It follows that $\bigcup K = \mu$. 
Hence in $V[G]$, $\mu$ is the union of $\omega_1$ many countable sets and therefore 
has size $\omega_1$.
\end{proof}

\begin{lemma}
Assume that $\omega_2 \le 2^\omega < \lambda$. 
Suppose that $R_0, \ldots, R_k$ are subsets of $H(\lambda)$. 
Then for any set $z \in H(\lambda)$, there are $M$ and $N$ in $\mathcal Y$ satisfying 
the following:
\begin{enumerate}
\item $z \in M \cap N$;
\item $\{ M, N \}$ is $(S)$ coherent adequate;
\item the structures $(M,\in,Y \cap M,R_0 \cap M,\ldots,R_k \cap M)$ 
and $(N,\in,Y \cap N,R_0 \cap N,\ldots,R_k \cap N)$ are 
elementary in $(H(\lambda),\in,Y,R_0,\ldots,R_k)$ and are isomorphic;
\item there exists $\alpha \in M \cap (2^\omega)^+$ and $\beta \in N \cap (2^\omega)^+$ such that 
$\alpha \ne \beta$ and $\sigma_{M,N}(\alpha) = \beta$. 
\end{enumerate}
\end{lemma}

\begin{proof}
For each $i < (2^\omega)^+$ fix $N_i$ in $\mathcal Y$ such that $z$ and $i$ are in $N_i$ and 
$N_i$ is an elementary substructure of $(H(\lambda),\in,Y,R_0,\ldots,R_k)$. 
This is possible since $\mathcal Y$ is stationary. 
Let $\mathfrak N_i$ denote the structure 
$(N_i,\in,Y \cap N_i,R_0 \cap N_i,\ldots,R_k \cap N_i)$ and let 
$\overline{\mathfrak N}_i$ denote its transitive collapse. 
Since $H(\omega_1)$ has size $2^\omega$, we can fix a cofinal set $P \subseteq (2^\omega)^+$ 
such that for all $i < j$ in $P$, $\overline{\mathfrak N}_i = \overline{\mathfrak N}_j$. 
It follows that $\mathfrak N_i$ and $\mathfrak N_j$ are isomorphic.

By the $\Delta$-system lemma, there is a cofinal set $P' \subseteq P$ and a countable set $z$ 
such that for all $i < j$ in $P'$, $N_i \cap N_j = z$. 
As there are $2^\omega$ many possibilities for $\sigma_{N_i} \restriction z$ for $i \in P'$, 
where $\sigma_{N_i}$ is the transitive collapsing map of $N_i$, 
we can find a cofinal set $P'' \subseteq P'$ such that for all $i < j$ in $P''$, 
$\sigma_{N_i} \restriction z = \sigma_{N_j} \restriction z$. 
It follows that if $a \in N_i \cap N_j = z$, then 
$\sigma_{N_i,N_j}(a) = \sigma_{N_j}^{-1}(\sigma_{N_i}(a)) = a$. 
So $N_i$ and $N_j$ are strongly isomorphic.

Let $M$ be equal to $N_i$ for some fixed $i \in P''$. 
Then $\sup(M \cap (2^\omega)^+) < (2^{\omega})^+$ since $M$ is countable. 
So we can find $i < j$ in $P''$ such that $\sup(M \cap (2^\omega)^+) < j$. 
Let $N := N_j$. 
We will prove that $M$ and $N$ are as desired. 

Properties (1) and (3) are immediate. 
Since $\omega_2 \le 2^\omega$ and $M$ and $N$ are isomorphic, 
$M \cap \omega_2 = N \cap \omega_2$ by Lemma 3.1. 
So trivially $\{ M, N \}$ is adequate. 
Also $R_M(N)$ and $R_N(M)$ are empty, so $\{ M, N \}$ is $(S)$ adequate. 
As $M$ and $N$ are strongly isomorphic, $\{ M, N \}$ is $(S)$ coherent adequate. 
This verifies property (2). 
For (4), let $\beta := j$. 
Then $\sup(M \cap (2^\omega)^+) < \beta$. 
Since $(2^\omega)^+$ is either equal to $\lambda$ or definable in $H(\lambda)$, 
$\alpha := \sigma_{N,M}(\beta)$ is in $M \cap (2^\omega)^+$ and hence is below $\beta$. 
So $\alpha \ne \beta$ and $\sigma_{M,N}(\alpha) = \beta$.
\end{proof}

\begin{proposition}
Let $\p$ be an $(S,\mathcal Y)$ coherent adequate type forcing. 
Let $\mu := (2^\omega)^+$. 
If $p$ is a condition which forces that $\langle f_i : i < \mu \rangle$ is a sequence of 
functions from $\omega$ to $\omega$, then there is $q \le p$ and $i < j$ such that 
$q$ forces that $\dot f_i = \dot f_j$.
\end{proposition}

The proof is nearly identical to the proof of Proposition 2.3, replacing $\omega_2^V$ with $\mu$ 
and replacing references to Lemma 2.2 with references to Lemma 3.3. 

\begin{corollary}
Let $\p$ be an $(S,\mathcal Y)$ coherent adequate type forcing. 
Then $\p$ collapses $(2^\omega)^V$ to have size $\omega_1$, 
forces that the successor of $(2^\omega)^V$ in $V$ is equal to $\omega_2$, 
and forces CH.
\end{corollary}

The proof of the corollary follows from Proposition 3.4 in the same way that Corollary 2.4 
follows from Proposition 2.3.

\section{Some Lemmas on Closure}

The next two sections will develop the technology needed to amalgamate 
coherent adequate style forcings over countable elementary substructures. 
There are several ways in which this development goes beyond the 
analogous results in \cite{jk23}. 
Besides the more general context of $H(\lambda)$, 
we also give an analysis of the remainder 
points produced under such amalgamation, and work out general 
results on closure. 
The present section handles the topic of 
closure.\footnote{Lemmas 4.1 and 4.2 are also true by the same arguments 
when coherent adequate sets 
are replaced by symmetric systems in the sense of \cite{mota}.}

Let $A$ be a subset of $\mathcal X$ and let $x$ be a finite subset of $H(\lambda)$. 
We say that the pair $(x,A)$ is \emph{closed} if whenever $N$ and $N'$ are 
isomorphic sets in $A$ and $a \in x \cap N$, then 
$\sigma_{N,N'}(a) \in x$.

\begin{lemma}
Let $A$ be a finite 
coherent adequate set and $x$ a finite subset of $H(\lambda)$. 
Let $y$ be the set 
$$
x \cup \{ \sigma_{M,M'}(a) : M, M' \in A, \ 
M \cong M', \ a \in x \cap M \}.
$$
Then $(y,A)$ is closed.
\end{lemma}

\begin{proof}
Let $N$ and $N'$ be isomorphic sets in $A$ and $a \in y \cap N$. 
We will prove that $\sigma_{N,N'}(a) \in y$. 
If $a \in x$, then $\sigma_{N,N'}(a) \in y$ by the definition of $y$. 
Otherwise there are $M$ and $M'$ in $A$ which are isomorphic and 
$b$ in $x \cap M$ such that $a = \sigma_{M,M'}(b)$. 
Then $a$ is in $M' \cap N$.

\bigskip

\emph{Case 1:} $M' < N$. 
Fix $N^* \in A$ isomorphic to $N$ with $M' \in N^*$. 
Then $a \in N^* \cap N$, so 
$\sigma_{N^*,N'}(a) = \sigma_{N,N'}(a)$. 
Let $M_1 = \sigma_{N^*,N'}(M')$. 
Then $M_1 \in A$ and $\sigma_{M',M_1} = \sigma_{N^*,N'} \restriction M'$ 
by Lemma 1.10. 
Now $\sigma_{M,M_1}(b)$ is in $y$ by the definition of $y$. 
But $\sigma_{M,M_1}(b) = \sigma_{M',M_1}(\sigma_{M,M'}(b)) = 
\sigma_{N^*,N'}(a) = \sigma_{N,N'}(a)$. 
Hence $\sigma_{N,N'}(a) \in y$.

\bigskip

\emph{Case 2:} $M' \cong N$. 
Since $a \in M' \cap N$, $\sigma_{M',N}(a) = a$. 
By the definition of $y$, $\sigma_{M,N'}(b) \in y$. 
But $\sigma_{M,N'}(b) = \sigma_{N,N'}(\sigma_{M',N}(\sigma_{M,M'}(b))) = 
\sigma_{N,N'}(\sigma_{M',N}(a)) = \sigma_{N,N'}(a)$. 
So $\sigma_{N,N'}(a) \in y$.

\bigskip

\emph{Case 3:} $N < M'$. 
Fix $M_0 \in A$ isomorphic to $M'$ such that $N \in M_0$. 
Let $N^* := \sigma_{M_0,M}(N)$, which is in $A$. 
Then $\sigma_{N^*,N} = \sigma_{M,M_0} \restriction N^*$ by Lemma 1.10. 
Since $a \in M' \cap N$, $a \in M' \cap M_0$. 
By Lemma 1.13, $\sigma_{M_0,M}(a) = \sigma_{M',M}(a) = b$. 
In particular, since $a \in N$, 
$b = \sigma_{M_0,M}(a) \in \sigma_{M_0,M}(N) = N^*$. 
By the definition of $y$, $\sigma_{N^*,N'}(b) \in y$. 
But $\sigma_{N^*,N'}(b) = \sigma_{N,N'}(\sigma_{N^*,N}(b)) = 
\sigma_{N,N'}(\sigma_{M,M_0}(b)) = \sigma_{N,N'}(a)$. 
So $\sigma_{N,N'}(a) \in y$.
\end{proof}

The next lemma analyzes closure in the context of 
amalgamation over countable models.

\begin{lemma}
Let $A$ be a coherent adequate set, $N \in A$, and suppose that $B$ is a coherent 
adequate set with $A \cap N \subseteq B \subseteq N$. 
Let $x$ and $y$ be subsets of $H(\lambda)$ such that $(x,A)$ is closed, 
$(y,B)$ is closed, and $x \cap N \subseteq y \subseteq N$. 
Define $C$ by 
$$
\{ M \in A : N \le M \} \cup 
\{ \sigma_{N,N'}(K) : N' \in A, \ N \cong N', \ K \in B \}.
$$
Define $z$ by 
$$
x \cup \{ \sigma_{N,N'}(a) : N' \in A, \ N \cong N', \ a \in y \}.
$$
Then the pair $(z,C)$ is closed, $x \cup y \subseteq z$, and $z \cap N = y$.
\end{lemma}

\begin{proof}
Assume that $K$ and $M$ are in $C$ and are isomorphic, 
and $a$ is in $z \cap K$. 
We will show that $\sigma_{K,M}(a) \in z$.

\bigskip

\emph{Case 1:} $N \le K$. 
Then $N \le M$. 
So both $K$ and $M$ are in $A$. 
If $a \in x$, then we are done since $(x,A)$ is closed. 
So assume that $a = \sigma_{N,N'}(a_0)$ 
for some $N'$ in $A$ isomorphic to $N$ and 
$a_0 \in y$.

\bigskip

\emph{Subcase 1a:} $K$ and $M$ are isomorphic to $N$. 
Since $a \in K \cap N'$, $\sigma_{N',K}(a) = a$ by Lemma 1.13. 
By the definition of $z$, $\sigma_{N,M}(a_0) \in z$. 
But $\sigma_{N,M}(a_0) = \sigma_{K,M}(\sigma_{N',K}((\sigma_{N,N'}(a_0))) = 
\sigma_{K,M}(\sigma_{N',K}(a)) = \sigma_{K,M}(a)$. 
So $\sigma_{K,M}(a) \in z$.

\bigskip

\emph{Subcase 1b:} $N < K$. 
Since $A$ is adequate, fix $J \in A$ isomorphic to $K$ such that 
$N' \in J$. 
Let $N'' := \sigma_{J,K}(N')$ and $N''' := \sigma_{K,M}(N'')$. 
Then $N''$ and $N'''$ are in $A$ and are isomorphic to $N$. 
Moreover $\sigma_{N',N''} = \sigma_{J,K} \restriction N'$ and 
$\sigma_{N'',N'''} = \sigma_{K,M} \restriction N''$. 
By definition of $z$, $a_1 := \sigma_{N,N'''}(a_0)$ is in $z$. 
Since $a$ is in $N' \cap K$, $a$ is in $J \cap K$, so $\sigma_{J,K}(a) = a$. 
Hence $a_1 = \sigma_{N,N'''}(a_0) = 
\sigma_{N'',N'''}(\sigma_{N',N''}(\sigma_{N,N'}(a_0))) = 
\sigma_{N'',N'''}(\sigma_{N',N''}(a)) = \sigma_{K,M}(\sigma_{J,K}(a)) = 
\sigma_{K,M}(a)$. 
So $\sigma_{K,M}(a) \in z$.

\bigskip

\emph{Case 2:} $K = \sigma_{N,N^*}(L)$ 
for some $N^* \in A$ isomorphic to $N$ and 
$L \in B$. 
Since $K$ and $M$ are isomorphic, $M \cap \omega_1 = K \cap \omega_1 < 
N \cap \omega_1$. 
So also $M = \sigma_{N,N'}(P)$ for some $N' \in A$ isomorphic to $N$ 
and $P \in B$. 
Then $L \cap \omega_1 = P \cap \omega_1$. 
Since $B$ is coherent adequate, 
$L$ and $P$ are isomorphic.

\bigskip

\emph{Subcase 2a:} 
$a = \sigma_{N,N''}(a_0)$ for some $N'' \in A$ isomorphic to $N$ 
and $a_0 \in y$. 
Since $a \in K = \sigma_{N,N^*}(L)$, 
$a \in N'' \cap N^*$. 
By Lemma 1.13, $\sigma_{N^*,N}(a) = \sigma_{N'',N}(a) = a_0$. 
So $\sigma_{K,L}(a) = \sigma_{N^*,N}(a) = a_0$. 
It follows that $\sigma_{K,M}(a) = 
\sigma_{P,M}(\sigma_{L,P}(\sigma_{K,L}(a))) = 
\sigma_{P,M}(\sigma_{L,P}(a_0)) = 
\sigma_{N,N'}(\sigma_{L,P}(a_0))$. 
Since $(y,B)$ is closed, 
$\sigma_{L,P}(a_0)$ is in $y$. 
So by the 
definition of $z$, $\sigma_{N,N'}(\sigma_{L,P}(a_0))$ is in $z$. 
That is, $\sigma_{K,M}(a)$ is in $z$.

\bigskip

\emph{Subcase 2b:} $a$ is in $x$. 
Since $a \in K = \sigma_{N,N^*}(L)$, it follows that $a \in N^*$. 
As $(x,A)$ is closed, $a_0 := \sigma_{N^*,N}(a)$ 
is in $x \cap N$ and hence in $y$. 
But now $a = \sigma_{N,N^*}(a_0)$, $N^*$ is in $A$, $N^*$ is 
isomorphic to $N$, and $a_0 \in y$. 
So we are done by subcase 2a.

\bigskip

Now let us prove that $x \cup y \subseteq z$ and $z \cap N = y$. 
The set $x$ is a subset of $z$ by definition. 
And if $a \in y$, then $a = \sigma_{N,N}(a)$ is in $z$. 
So $x \cup y \subseteq z$. 
The set $y$ is a subset of $N$ by definition, and we just showed 
that $y$ is a subset of $z$. 
So $y \subseteq z \cap N$. 
On the other hand, suppose that $a \in z \cap N$. 
If $a \in x$, then $a \in x \cap N$ and hence $a \in y$. 
Otherwise $a = \sigma_{N,N'}(a_0)$, where $N' \in A$ 
is isomorphic to $N$ and $a_0 \in y$. 
But then $a \in N' \cap N$, so $\sigma_{N',N}(a) = a$. 
Hence $a = a_0$, so $a \in y$.
\end{proof}

\section{Amalgamating over Countable Models}

In this section we prove an amalgamation result for countable models. 
The difference between this proposition and the analogous result from 
\cite{jk23} is that here we provide an analysis of remainder points.

The next lemma considers a special case of the general result.

\begin{lemma}
Let $N$, $N'$, $N^*$ be isomorphic in $\mathcal X$ such that 
$\{ N, N', N^* \}$ is coherent adequate. 
Assume that $K$ and $L$ are in $N \cap \mathcal X$ and 
$\{ K, L \}$ is adequate. 
Let $M := \sigma_{N,N'}(K)$ and $P := \sigma_{N,N^*}(L)$. 
Then $\{ M, P \}$ is adequate. 
Moreover, if $\zeta$ is in $R_P(M) \cap \beta_{N',N^*}$, then 
$\sigma_{N',N}(\zeta)$ is in $R_L(K)$.
\end{lemma}

\begin{proof}
Let $\sigma := \sigma_{N',N} \restriction N' \cap N^*$. 
By Lemma 1.13, $\sigma$ is equal to $\sigma_{N^*,N} \restriction N' \cap N^*$.

Note that $\beta_{M,P} \le \beta_{N',N^*}$, since $M \in N'$ and 
$P \in N^*$. 
Therefore $M \cap \beta_{M,P}$ is a countable subset of $\beta_{N',N^*}$. 
So $M \cap \beta_{M,P}$ is in $N' \cap P_{\omega_1}(\beta_{N',N^*})$. 
As $N'$ and $N^*$ are strongly isomorphic, $M \cap \beta_{M,N} \in N^*$ 
by Lemma 1.18. 
For the same reason, $P \cap \beta_{M,P}$ is in $N'$. 
Since $\beta_{M,P}$ is definable from $M \cap \beta_{M,P}$ as 
$\min(\Lambda \setminus \sup(M \cap \beta_{M,P}))$, 
we have that $\beta_{M,P} \in N' \cap N^*$. 
Also $\sigma(M \cap \beta_{M,P}) = 
\sigma_{N',N}(M) \cap \sigma(\beta_{M,P}) = K \cap \sigma(\beta_{M,P})$, 
and similarly $\sigma(P \cap \beta_{M,P}) = L \cap \sigma(\beta_{M,P})$.

\bigskip

We claim that $\sigma(\beta_{M,P}) \le \beta_{K,L}$. 
Since $\beta_{M,P} \in \Lambda_{M}$, $\sigma(\beta_{M,P}) \in 
\sigma_{N',N}(\Lambda_M) = \Lambda_{K}$. 
Similarly, $\sigma(\beta_{M,P}) \in \Lambda_{L}$. 
Since $\beta_{K,L} = \max(\Lambda_K \cap \Lambda_L)$, 
$\sigma(\beta_{M,P}) \le \beta_{K,L}$.

\bigskip

Let us show that $\{ M, P \}$ is adequate. 
We will use the fact that $\{ K, L \}$ is adequate and consider three cases.

\bigskip

\emph{Case 1:} $K < L$. 
Then $K \cap \beta_{K,L} \in L$. 
Since $\sigma(\beta_{M,P}) \le \beta_{K,L}$, 
$K \cap \sigma(\beta_{M,P}) \in L$. 
But $\sigma(M \cap \beta_{M,P}) = K \cap \sigma(\beta_{M,P})$ 
and $\sigma = \sigma_{N^*,N} \restriction N' \cap N^*$. 
Applying $\sigma_{N,N^*}$, we get that 
$\sigma^{-1}(K \cap \sigma(\beta_{M,N})) \in \sigma_{N,N^*}(L)$, that is  
$M \cap \beta_{M,P} \in P$.

\bigskip

\emph{Case 2:} $L < K$. 
This case follows by a symmetric argument.

\bigskip

\emph{Case 3:} $K \sim L$. 
As $\sigma(\beta_{M,P}) \le \beta_{K,L}$, it follows that 
$K \cap \sigma(\beta_{M,P}) = L \cap \sigma(\beta_{M,P})$. 
As noted above, $\sigma(M \cap \beta_{M,P}) = K \cap \sigma(\beta_{M,P})$ and 
$\sigma(P \cap \beta_{M,P}) = L \cap \sigma(\beta_{M,P})$. 
So applying $\sigma^{-1}$, we get that 
$M \cap \beta_{M,P} = P \cap \beta_{M,P}$.

\bigskip

This completes the proof that $\{ M, P \}$ is adequate. 
Now let $\zeta$ be in $R_P(M) \cap \beta_{N',N^*}$. 
We will show that $\zeta_0 := \sigma(\zeta)$ is in $R_L(K)$. 
Since $\zeta \in M$, applying $\sigma_{N',N}$, it follows that $\zeta_0 \in K$. 
Note that since $\beta_{M,P} \le \zeta$, 
$\sigma(\beta_{M,P}) \le \zeta_0$. 
So $\zeta_0 \in K \setminus \sigma(\beta_{M,P})$.

\bigskip

\emph{Case A:} $M \le P$. 
Then $K \le L$. 
We claim that in this case, $\sigma(\beta_{M,P}) = \beta_{K,L}$. 
We already know that $\sigma(\beta_{M,P}) \le \beta_{K,L}$. 
Suppose for a contradiction that $\sigma(\beta_{M,P}) < \beta_{K,L}$. 
By Lemma 1.2, $K \cap 
[\sigma(\beta_{M,P}),\beta_{K,L})$ is nonempty. 
So letting $\tau_0 = \min(K \setminus \sigma(\beta_{M,P}))$, 
$\tau_0 < \beta_{K,L}$. 
Since $\sigma(\zeta) = \zeta_0 \in K \setminus \sigma(\beta_{M,P})$, 
$\tau_0 \le \zeta_0$ by the minimality of $\tau_0$. 
So $\tau := \sigma_{N,N'}(\tau_0) \le \zeta$.
Also since $\tau_0 \in K$, $\tau \in M$. 
And because $\sigma(\beta_{M,P}) \le \tau_0$, 
$\beta_{M,P} \le \tau$. 
So $\tau$ is in $M \setminus \beta_{M,P}$.

Since $K \le L$, $\tau_0 \in K \cap \beta_{K,L}$ implies that $\tau_0 \in L$. 
So $\tau_0 \in K \cap L$. 
As $\tau \le \zeta$ and $\zeta < \beta_{N',N^*}$, 
$\tau < \beta_{N',N^*}$. 
So $\tau = \sigma_{N,N'}(\tau_0) = \sigma_{N,N^*}(\tau_0)$ by Lemma 1.13. 
Since $\tau_0 \in L$, $\tau = \sigma_{N,N^*}(\tau_0) \in P$. 
So $\tau \in (M \cap P \cap \omega_2) \setminus \beta_{M,P}$, 
which is impossible. 
This contradiction shows that $\sigma(\beta_{M,P}) = \beta_{K,L}$.

Now we show that $\zeta_0 \in R_L(K)$. 
If $\zeta = \min(M \setminus \beta_{M,P})$, then applying 
$\sigma_{N',N}$ we get that 
$\zeta_0 = \min(K \setminus \sigma(\beta_{M,P})) = 
\min(K \setminus \beta_{K,L})$. 
So $\zeta_0 \in R_{L}(K)$. 
Now assume that there is $\gamma \in P \setminus \beta_{M,P}$ 
such that $\zeta = \min(M \setminus \gamma)$. 
Since $\zeta < \beta_{N',N^*}$, $\gamma < \beta_{N',N^*}$. 
Let $\gamma_0 := \sigma(\gamma)$. 
Applying $\sigma_{N',N}$, $\zeta_0 = \min(K \setminus \gamma_0)$. 
As $\sigma(\beta_{M,P}) = \beta_{K,L}$ and $\beta_{M,P} \le \gamma$,  
$\beta_{K,L} \le \gamma_0$. 
Also $\gamma \in P$ implies that 
$\gamma_0 = \sigma_{N^*,N}(\gamma) \in 
\sigma_{N^*,N}(P) = L$. 
So $\gamma_0 \in L \setminus \beta_{K,L}$. 
Therefore $\zeta_0 \in R_L(K)$.

\bigskip

\emph{Case 2:} $P < M$. 
Then $L < K$. 
As $\zeta \in R_P(M)$, there is $\gamma \in P \setminus \beta_{M,P}$ 
such that $\zeta = \min(M \setminus \gamma)$. 
Since $\gamma < \zeta < \beta_{N',N^*}$, 
$\gamma \in P \cap \beta_{N',N^*}$. 
Applying $\sigma_{N^*,N}$, $\gamma_0 := \sigma(\gamma)$ is in $\sigma_{N^*,N}(P) = L$. 
Since $\gamma \in P \setminus \beta_{M,P}$, 
$\gamma$ is not in $M$. 
Applying $\sigma_{N',N}$, $\gamma_0$ 
is not in $\sigma_{N',N}(M) = K$. 
So $\gamma_0 \in L \setminus K$. 
Since $L < K$, this implies that $\beta_{K,L} \le \gamma_0$. 
Therefore $\gamma_0 \in L \setminus \beta_{K,L}$. 
Since $\zeta = \min(M \setminus \gamma)$, 
$\zeta_0 = \min(K \setminus \gamma_0)$. 
Hence $\zeta_0 \in R_L(K)$.
\end{proof}

\begin{proposition}
Let $A$ be a coherent adequate set and $N \in A$. 
Suppose that $B$ is a coherent adequate set and 
$A \cap N \subseteq B \subseteq N$. 
Let $C$ be the set 
$$
\{ M \in A : N \le M \} \cup 
\{ \sigma_{N,N'}(K) : N' \in A, \ N \cong N', \ 
K \in B \}.
$$
Then $C$ is a coherent adequate set, $A \cup B \subseteq C$, 
and $C \cap N = B$. 
Moreover, $R_C$ is a subset of 
$$
R_A \cup \{ \min(K \setminus \zeta) : K \in C, \ \zeta \in R_A \} \cup 
\{ \sigma_{N,N'}(\tau) : N' \in A, \ N \cong N', \ \tau \in R_B \}.
$$
\end{proposition}

\begin{proof}
Let $M$ and $P$ be in $C$ such that $M \cap \omega_1 \le P \cap \omega_1$. 
We will show that $\{ M, P \}$ is adequate and that the remainder points of $M$ and $P$ 
are as required.

\bigskip

\emph{Case 1:} $N \cap \omega_1 \le M \cap \omega_1$.  
Then $N \cap \omega_1 \le M \cap \omega_1 \le P \cap \omega_1$. 
Hence $M$ and $P$ are both in $A$. 
So obviously $\{ M, P \}$ is adequate and 
$R_M(P) \cup R_P(M) \subseteq R_A$.

\bigskip

\emph{Case 2:} $M \cap \omega_1 < N \cap \omega_1 \le P \cap \omega_1$. 
Then clearly $P \in A$, $N \le P$, and 
$M = \sigma_{N,N'}(K)$ for some $N' \in A$ isomorphic to 
$N$ and $K \in B$. 
Since $N \le P$ and $N$ and $N'$ are isomorphic, $N' \le P$. 
And as $M \in N'$, $\beta_{M,P} \le \beta_{N',P}$. 
So $M \cap \beta_{M,P}$ is in $N' \cap P_{\omega_1}(\beta_{N',P})$. 
As $N' \le P$, $M \cap \beta_{M,P} \in P$ by Lemma 1.19. 
So $M < P$.

Let $\zeta \in R_M(P)$ be given, and we will show that $\zeta \in R_A$. 
Since $M < P$, there is $\gamma \in M \setminus \beta_{M,P}$ such that 
$\zeta = \min(P \setminus \gamma)$. 
As $\beta_{M,P} \le \gamma$, $\gamma$ is not in $P$. 
Since $N' \le P$ and $\gamma \in N' \setminus P$, 
$\beta_{N',P} \le \gamma$. 
So $\gamma \in N' \setminus \beta_{N',P}$. 
Therefore $\zeta$ is in $R_{N'}(P)$ and hence in $R_A$.

Now let $\zeta \in R_P(M)$ be given. 
Since $M < P$, there is $\gamma \in P \setminus \beta_{M,P}$ such that 
$\zeta = \min(M \setminus \gamma)$. 
As $\zeta$ is in $N' \setminus P$ and $N' \le P$, 
$\beta_{N',P} \le \zeta$. 
If $\gamma < \beta_{N',P}$, then $\zeta = \min(M \setminus \beta_{N',P})$. 
Let $\tau := \min(N' \setminus \beta_{N',P})$. 
Then $\tau \in R_A$. 
Since $M \subseteq N'$, $\zeta = \min(M \setminus \tau)$. 
Thus $\zeta$ is as required. 
Now assume that $\beta_{N',P} \le \gamma$. 
Then $\gamma$ is in $P \setminus \beta_{N',P}$. 
Hence $\pi := \min(N' \setminus \gamma)$ is in $R_A$. 
But then $\zeta = \min(M \setminus \pi)$, so $\zeta$ is as required.

\bigskip

\emph{Case 3:} $M \cap \omega_1 \le P \cap \omega_1 < N \cap \omega_1$. 
Then $M = \sigma_{N,N'}(K)$ and $P = \sigma_{N,N^*}(L)$ for some 
$N'$ and $N^*$ in $A$ isomorphic to $N$ and 
$K$ and $L$ in $B$. 
By Lemma 5.1, $\{ M, P \}$ is adequate. 
Moreover, if $\zeta$ is in $R_P(M) \cap \beta_{N',N^*}$ then 
$\sigma_{N',N}(\zeta) \in R_B$, and if $\zeta$ is in $R_M(P) \cap \beta_{N',N^*}$ 
then $\sigma_{N^*,N}(\zeta) \in R_B$. 
In either case, $\zeta$ is as required.

Let $\zeta \in R_P(M) \setminus \beta_{N',N^*}$ be given. 
Since $M \in N'$ and $P \in N^*$, $\beta_{M,P} \le \beta_{N',N^*}$. 
If $\zeta = \min(M \setminus \beta_{M,P})$ or if 
$\zeta = \min(M \setminus \gamma)$ 
for some $\gamma \in P \setminus \beta_{M,P}$ 
which is below $\beta_{N',N^*}$, then 
$\zeta = \min(M \setminus \beta_{N',N^*})$. 
So letting $\tau := \min(N' \setminus \beta_{N',N^*})$, $\tau \in R_A$ 
and clearly $\zeta = \min(M \setminus \tau)$. 
So $\zeta$ is as required. 
Otherwise there is $\gamma \in P \setminus \beta_{M,P}$ such that 
$\beta_{N',N^*} \le \gamma$ and $\zeta = \min(M \setminus \gamma)$. 
Since $\gamma \in P$ and $P \in N^*$, $\gamma \in N^* \setminus \beta_{N',N^*}$. 
Let $\pi := \min(N' \setminus \gamma)$. 
Then $\pi \in R_A$ and clearly $\zeta = \min(M \setminus \pi)$, and we are done. 
The proof that the remainder points in $R_M(P)$ are as required follows 
by a symmetric argument, since we never used in this paragraph 
the assumption that 
$M \cap \omega_1 \le P \cap \omega_1$.

\bigskip

Now we show that $A \cup B \subseteq C$. 
If $K \in B$, then $K = \sigma_{N,N}(K)$ is in $C$. 
Let $M \in A$. 
If $N \cap \omega_1 \le M \cap \omega_1$, then $M \in C$. 
Otherwise $M \cap \omega_1 < N \cap \omega_1$. 
Since $A$ is coherent, there exists $N' \in A$ 
isomorphic to $N$ such that $M \in N'$. 
Let $K := \sigma_{N',N}(M)$. 
Then $K$ is in $A \cap N$ and hence in $B$. 
So $\sigma_{N,N'}(K) = M$ is in $C$.

\bigskip

It remains to prove that $C$ is coherent and $C \cap N = B$. 
We apply Lemma 4.2 in the case where $x = A$ and $y = B$. 
Then clearly the set $z$ defined there is equal to $C$. 
Since $A$ and $B$ are coherent, the pairs $(A,A)$ and $(B,B)$ are closed. 
So by Lemma 4.2, $(C,C)$ is closed and $C \cap N = B$. 
Therefore if $M$ and $M'$ are in $C$ and are isomorphic, and $K \in M \cap C$, 
then $\sigma_{M,M'}(K) \in C$.

We prove the remaining properties in the definition of coherence. 
Suppose that $M$ and $P$ are in $C$ and 
$M \cap \omega_1 = P \cap \omega_1$. 
If $N \le M$, then $M$ and $P$ are both in $A$ and hence are 
strongly isomorphic. 
Otherwise $M < N$ and $P < N$, which implies that 
$M = \sigma_{N,N'}(K)$ and $P = \sigma_{N,N^*}(L)$ for some 
$N'$ and $N^*$ in $A$ isomorphic to $N$ and 
$K$ and $L$ in $B$. 
Then $K \cap \omega_1 = M \cap \omega_1 = P \cap \omega_1 = L \cap \omega_1$. 
Since $B$ is coherent, $K$ and $L$ are strongly isomorphic. 
By Lemma 1.14, $M$ and $P$ are strongly isomorphic.

Now assume that $M$ and $P$ are in $C$ and $M < P$. 
We will show that there is $P'$ in $C$ isomorphic to $P$ 
such that $M \in P'$. 
If $N \cap \omega_1 \le M \cap \omega_1$, then $M$ and $P$ 
are both in $A$, and we are done since $A$ is coherent. 
Suppose that $M \cap \omega_1 < N \cap \omega_1 \le P \cap \omega_1$. 
Then $M = \sigma_{N,N'}(K)$ for some $N' \in A$ 
isomorphic to $N$ and $K \in B$. 
If $N'$ is isomorphic to $P$ then we are done. 
Otherwise $N' < P$, so there is $P' \in A$ isomorphic to $P$ with $N' \in P'$. 
Then $M \in P'$.

Finally, assume that $M = \sigma_{N,N'}(K)$ and 
$P = \sigma_{N,N^*}(L)$ for some $K$ and $L$ in $B$ 
and $N'$ and $N^*$ in $A$ isomorphic to $N$. 
Since $M < P$, $K < L$. 
As $B$ is coherent, fix $L'$ in $B$ isomorphic to $L$ such that $K \in L'$. 
Then $M = \sigma_{N,N'}(K) \in \sigma_{N,N'}(L')$ and 
$\sigma_{N,N'}(L') \in C$. 
Since $L$ and $L'$ are strongly isomorphic in $N$, 
$\sigma_{N,N'}(L')$ and $\sigma_{N,N^*}(L) = P$ are strongly isomorphic 
by Lemma 1.14.
\end{proof}

\section{Adding a square sequence}

We review the forcing poset from \cite{jk23} for adding a square sequence with 
finite conditions, and show that it is in the class of coherent adequate type 
forcing posets. 
As a consequence, this forcing poset preserves $CH$.

By a \emph{triple} we mean a sequence 
$\langle \alpha, \gamma, \beta \rangle$, where $\alpha \in \Lambda$ and 
$\gamma < \beta < \alpha$. 
Given distinct triples $\langle \alpha, \gamma, \beta \rangle$ and 
$\langle \alpha', \gamma', \beta' \rangle$, we say that the triples are 
\emph{nonoverlapping} if either $\alpha \ne \alpha'$, 
or $\alpha = \alpha'$ and $[\gamma,\beta) \cap [\gamma',\beta') = \emptyset$; 
otherwise they are \emph{overlapping}. 

\begin{definition}
Let $\p$ be the forcing poset whose conditions are pairs 
$( x, A )$ satisfying:
\begin{enumerate}
\item $x$ is a finite set of nonoverlapping triples;
\item $A$ is a finite coherent adequate set;
\item for all $M \in A$ and $\langle \alpha, \gamma, \beta \rangle \in x$ 
such that $\alpha \in M$, either $\gamma$ and $\beta$ are in $M$ 
or $\sup(M \cap \alpha) < \gamma$;
\item if $M$ and $M'$ are isomorphic sets in $A$, 
then for any triple $\langle \alpha, \gamma, \beta \rangle \in M \cap x$, 
$\sigma_{M,M'}(\langle \alpha, \gamma, \beta \rangle) \in x$.
\end{enumerate}
Let $(y,B) \le (x,A)$ if $x \subseteq y$ and $A \subseteq B$.
\end{definition}

We proved in \cite{jk23} that $\p$ is strongly proper, $\omega_2$-c.c., 
and forces $\Box_{\omega_1}$.

\begin{proposition}
The forcing poset $\p$ is a coherent adequate type forcing poset.
\end{proposition}

\begin{proof}
Let $(x,A)$ be condition in $\p$ and we will 
verify requirements (I)--(V) in the 
definition of a coherent adequate type forcing. 
The proof of Proposition 4.3 in \cite{jk23} shows that if $(x,A)$ is a 
condition and $M \in A$, then $(x,A)$ is strongly $M$-generic. 
So requirement (V) holds. 
Obviously $x$ is a subset of $H(\omega_2)$, and by assumption $A$ is a 
coherent adequate set. 
Thus requirements (I) and (II) are immediate.

(III) Suppose that $(y,B) \le (x,A)$, $N$ and $N'$ are isomorphic sets in $B$, 
and $(x,A) \in N$. 
Properties (1)--(4) in the definition of $\p$ are all first order 
definable in $H(\omega_2)$, except the part of (2) which asserts membership 
in $\mathcal X$. 
Since $\sigma_{N,N'}$ is an isomorphism which preserves membership 
in $\mathcal X$ and $N$ and $N'$ are 
elementary substructures of $H(\omega_2)$, 
$\sigma_{M,N}((x,A))$ is in $\p$. 
Also $\sigma_{M,N}((x,A)) = (\sigma_{M,N}[x],\sigma_{M,N}[A])$. 
Since $(y,B)$ is closed and $B$ is coherent, $x \subseteq y$ and 
$A \subseteq B$ imply that 
$\sigma_{M,N}[x] \subseteq y$ and $\sigma_{M,N}[A] \subseteq B$. 
Consequently, $(y,B) \le \sigma_{M,N}((x,a))$.

(IV) Suppose that $M_0, \ldots, M_n$ 
are isomorphic sets in $\mathcal X$ such that 
$\{ M_0, \ldots, M_n \}$ is coherent adequate and 
$(x,A) \in M_0 \cap \cdots \cap M_n$. 
We claim that $(x,A \cup \{ M_0, \ldots, M_n \})$ satisfies properties 
(1)--(4) in the definition of $\p$. 
By Lemma 1.17, $A \cup \{ M_0, \ldots, M_n \}$ 
is coherent adequate so (2) is satisfied. 
(1) is immediate. 
For (3), any triple in $x$ is a member of 
$M_0, \ldots, M_n$, so there is nothing to check. 
And for (4), if $i, j \le n$ and $a \in x \cap M_i$, then 
$a \in M_i \cap M_j$ so $\sigma_{M_i,M_j}(a) = a$ is in $x$.
\end{proof}

\section{Adding a club preserving CH}

A forcing poset for adding a club to a fat subset of $\omega_2$ using 
adequate sets as side conditions appears in \cite{jk24}. 
Friedman \cite{friedman} asked whether it is possible to add a club to 
$\omega_2$ with finite conditions while preserving CH. 
In this section we adapt of variation of the poset from \cite{jk24} to 
solve this problem.

Fix a stationary set $S \subseteq \omega_2$ which is fat. 
So for every club set $C \subseteq \omega_2$, 
$S \cap C$ contains a closed subset with order type $\omega_1 + 1$. 
Without loss of generality we may assume that 
$S \cap \cof(\omega_1)$ is stationary and for all 
$\alpha \in S \cap \cof(\omega_1)$, $S \cap \alpha$ 
contains a closed cofinal subset of $\alpha$. 
For the assumption of $S$ being fat implies that there is a stationary subset 
of $S$ satisfying this property.

The general framework of coherent adequate sets introduced in the 
first section involves the parameters 
$\lambda$ and $Y$. 
For this application we let $\lambda := \omega_2$ and let $Y$ code $S$ together with 
a well-ordering of $H(\omega_2)$. 
So $\mathcal X$ consists of countable elementary substructures 
of $H(\omega_2)$, and isomorphisms between members of 
$\mathcal X$ preserve membership in $S$.

Let $\mathcal Y$ denote the set of $M$ in $\mathcal X$ such that 
for all $\alpha \in (M \cap S) \cup \{ \omega_2 \}$, 
$\sup(M \cap \alpha) \in S$. 
A straightforward argument using the properties of $S$ shows that 
$\mathcal Y$ is stationary in $P_{\omega_1}(H(\omega_2))$. 
Also since isomorphisms between members of $\mathcal X$ preserve 
membership in $S$, it is easy to check that $\mathcal Y$ is closed 
under isomorphisms.

For an ordinal $\alpha$ and a set $N$ 
such that $N \cap \omega_2$ is not a subset of $\alpha$, 
let $\alpha_N$ denote $\min(N \setminus \alpha)$. 
Given pairs of ordinals $\langle \alpha, \alpha' \rangle$ and 
$\langle \gamma, \gamma' \rangle$, we say that the pairs \emph{overlap} if 
$\alpha < \gamma \le \alpha'$ or $\gamma < \alpha \le \gamma'$; 
otherwise they are \emph{nonoverlapping}.

\begin{definition}
Let $\p$ be the forcing poset consisting of conditions of the form 
$p = (x_p, A_p)$ satisfying:
\begin{enumerate}
\item $x_p$ is a finite set of nonoverlapping pairs of the form 
$\langle \alpha, \alpha' \rangle$, where 
$\alpha \le \alpha' < \omega_2$ and $\alpha \in S$;
\item $A_p$ is a finite coherent adequate subset of $\mathcal Y$;
\item if $\langle \alpha, \alpha' \rangle \in x_p$, $N \in A_p$, and 
$N \cap \omega_2 \nsubseteq \alpha$, then $N \cap [\alpha,\alpha'] \ne \emptyset$ 
implies that $\alpha$ and $\alpha'$ are in $N$, and 
$N \cap [\alpha,\alpha'] = \emptyset$ implies that 
$\langle \alpha_N, \alpha_N \rangle \in x_p$;
\item for all $\zeta$ in $R_{A_p}$, $\langle \zeta, \zeta \rangle \in x$;
\item if $M$ and $N$ are in $A_p$ and are isomorphic, then for any 
$a \in M \cap x_p$, $\sigma_{M,N}(a) \in x_p$.
\end{enumerate}
Let $q \le p$ if $x_p \subseteq x_q$ and $A_p \subseteq A_q$.
\end{definition}

Observe that if $(x,A)$ is a condition and $N \in A$, then 
$(x \cap N, A \cap N)$ is also a condition. 
Also, if $(x,A)$ is in $\p$ and $B$ is a subset of $A$ which 
is coherent adequate, 
then $(x,B)$ is a condition.

Note that in requirement (3), if the pair is of the form $\langle \alpha, \alpha \rangle$, then 
the conclusion in either case is equivalent to requiring that $\langle \alpha_N, \alpha_N \rangle \in x$.

Let $\dot C_S$ be a $\p$-name such that $\p$ forces
$$
\dot C_S  = \{ \alpha : \exists p \in \dot G \ \exists \alpha' \ 
\langle \alpha, \alpha' \rangle \in x_p \}.
$$
Clearly $\dot C_S$ is forced to be a subset of $S$. 
We will show that $\p$ is an $(S,\mathcal Y)$ coherent adequate type forcing and $\p$ forces 
that $\dot C_S$ is club in $\omega_2$.

\begin{proposition}
Let $q = (x,A)$ be a condition and assume that $N \in A$. 
Then $q$ is strongly $N$-generic.
\end{proposition}

\begin{proof}
Fix a set $D$ which is dense in the poset $N \cap \p$, and we will 
show that $D$ is predense below $q$.
Let $r \le q$ be given. 
Then $r \restriction N := (x_r \cap N, A_r \cap N)$ is a condition in $N$. 
Since $D$ is dense, fix $w$ in $D$ below $r \restriction N$. 
We will prove that $r$ and $w$ are compatible.

Let $C$ denote the set 
$$
\{ M \in A_r : N \le M \} \cup 
\{ \sigma_{N,N'}(K) : N' \in A_r, \ N \cong N', \ 
K \in A_w \}.
$$
Let $z$ denote the set 
$$
x_r \cup 
\{ \sigma_{N,N'}(a) : N' \in A_r, \ N \cong N', \ a \in x_w \}.
$$
Let $s := (z, C)$. 
We will show that $s$ is a condition and $s \le r, w$.

By Proposition 5.2, 
$C$ is a finite coherent adequate set, $A_r \cup A_w \subseteq C$, 
and $C \cap N = A_w$. 
Also $C$ is a subset of $\mathcal Y$ since $\mathcal Y$ is closed under isomorphisms. 
So $s$ satisfies requirement (2) in the definition of $\p$. 
By Lemma 4.2, the pair $(z,C)$ is closed, $x_r \cup x_w \subseteq z$, 
and $z \cap N = x_w$. 
Since $(z,C)$ is closed, if $M$ and $M'$ are isomorphic sets in $C$, 
then for any $a \in M \cap z$, $\sigma_{M,M'}(a) \in z$. 
Thus $s$ satisfies requirement (5) in the definition of $\p$. 
Since $x_r \cup x_w \subseteq z$ and $A_r \cup A_w \subseteq C$, 
it follows that if $s$ is a condition, then $s \le r, w$. 

It remains to show that $s$ satisfies requirements (1), (3), and (4) 
in the definition of $\p$. 
Regarding (1), it is easy to see that 
$z$ consists of pairs $\langle \alpha, \alpha' \rangle$ where 
$\alpha \le \alpha' < \omega_2$ and $\alpha \in S$, since this is true of 
pairs in $x_r$ and $x_w$ and these properties are 
preserved under isomorphisms. 
The proof that $z$ consists of nonoverlapping pairs will use requirement (3), 
so we verify (3) first.

\bigskip

(3) Let $\langle \alpha, \alpha' \rangle \in z$ and $M \in C$ be given, 
and assume that $M \nsubseteq \alpha$. 
We will show that $M \cap [\alpha,\alpha'] \ne \emptyset$ implies that 
$\alpha$ and $\alpha'$ are in $M$, and 
$M \cap [\alpha,\alpha'] = \emptyset$ implies that 
$\langle \alpha_M, \alpha_M \rangle \in z$.

\bigskip

\emph{Case 1:} $N \le M$ and $\langle \alpha, \alpha' \rangle \in x_r$. 
Then $M$ is in $A_r$. 
So we are done since $r$ is a condition.

\bigskip

\emph{Case 2:} $N \le M$ and $\langle \alpha, \alpha' \rangle = 
\sigma_{N,N'}(\langle \alpha_0, \alpha_0' \rangle)$ for some 
$\langle \alpha_0, \alpha_0' \rangle \in x_w$ and 
$N' \in A_r$ isomorphic to $N$. 
Since $N \le M$, $N' \le M$. 
If $\alpha' < \beta_{M,N'}$, then 
$N' \cap \beta_{M,N'} \subseteq M$ implies that 
$\alpha$ and $\alpha'$ are in $M$. 
So assume that $\beta_{M,N'} \le \alpha'$.

We claim that $R_{M}(N') \cap (\alpha,\alpha'] = \emptyset$. 
Otherwise assume that $\tau$ is in this intersection. 
Since $r$ is a condition, $\langle \tau, \tau \rangle$ is in $x_r \cap N'$.  
Let $\tau_0 := \sigma_{N',N}(\tau)$. 
Then $\langle \tau_0, \tau_0 \rangle$ is in $x_r \cap N$ and 
hence in $x_w$. 
Since $\sigma_{N',N}$ is order preserving, 
$\alpha_0 < \tau_0 \le \alpha_0'$, which contradicts that $w$ is a condition.

If $\alpha < \beta_{M,N'} \le \alpha'$, then $\min(N' \setminus \beta_{M,N'})$ 
is in $R_{M}(N') \cap (\alpha,\alpha']$, contradicting the claim. 
Therefore $\beta_{M,N'} \le \alpha$. 
If $\xi \in M \cap [\alpha,\alpha']$, then $\xi \in M \setminus \beta_{M,N'}$. 
So $\min(N' \setminus \xi)$ is in $R_M(N') \cap (\alpha,\alpha']$, 
contradicting the claim. 
So $M \cap [\alpha,\alpha'] = \emptyset$. 
Since $\alpha \in N' \setminus \beta_{M,N'}$, 
$\alpha_M = \min(M \setminus \alpha)$ is in $R_{N'}(M)$. 
Hence $\langle \alpha_M, \alpha_M \rangle \in x_r$.

\bigskip

\emph{Case 3:} $M = \sigma_{N,N'}(K)$ for some $K \in A_w$ and 
$N'$ in $A_r$ isomorphic to $N$ and $N' \cap [\alpha,\alpha'] \ne \emptyset$. 
Then $\alpha$ and $\alpha'$ are in $N'$ by cases 1 and 2. 
Let $\alpha_0 := \sigma_{N',N}(\alpha)$ and 
$\alpha_0' := \sigma_{N',N}(\alpha')$. 
Since $s$ satisfies requirement (5), 
$\langle \alpha_0, \alpha_0' \rangle$ is in $z \cap N$. 
Since $z \cap N = x_w$, $\langle \alpha_0, \alpha_0' \rangle \in x_w$.

Assume that $M \cap [\alpha,\alpha'] \ne \emptyset$. 
Applying $\sigma_{N',N}$, $K \cap [\alpha_0,\alpha_0'] \ne \emptyset$. 
Since $w$ is a condition, $\alpha_0$ and $\alpha_0'$ are in $K$. 
Therefore their images under $\sigma_{N,N'}$, namely $\alpha$ and $\alpha'$, 
are in $\sigma_{N,N'}(K) = M$. 

Now assume that $M \cap [\alpha,\alpha'] = \emptyset$. 
Applying $\sigma_{N',N}$, 
$K \cap [\alpha_0,\alpha_0'] = \emptyset$. 
Also $\sigma_{N',N}(\alpha_M) = \sigma_{N',N}(\min(M \setminus \alpha)) = 
\min(K \setminus \alpha_0) = (\alpha_0)_K$. 
Since $w$ is a condition, $\langle (\alpha_0)_K, (\alpha_0)_K \rangle \in x_w$. 
Therefore $\sigma_{N,N'}(\langle (\alpha_0)_K, (\alpha_0)_K \rangle) = 
\langle \alpha_M, \alpha_M \rangle$ is in $z$.

\bigskip

\emph{Case 4:} $M = \sigma_{N,N'}(K)$ for some $K \in A_w$ and 
$N'$ in $A_r$ isomorphic to $N$ and $N' \cap [\alpha,\alpha'] = \emptyset$. 
As $M \in N'$, clearly $M \cap [\alpha,\alpha'] = \emptyset$. 
Since $\alpha_M$ exists, $\tau := \min(N' \setminus \alpha)$ exists. 
As $N \le N'$, $\langle \tau, \tau \rangle \in z$ by Cases 1 and 2. 
If $\tau = \alpha_M$ then we are done. 
Otherwise $\tau < \alpha_M$ and $\alpha_M = \min(M \setminus \tau)$. 
Let $\tau_0 = \sigma_{N',N}(\tau)$. 
Since $s$ satisfies requirement (5), 
$\langle \tau_0, \tau_0 \rangle$ is in $z \cap N = x_w$. 
Also $\min(K \setminus \tau_0)$ is equal to 
$\sigma_{N',N}(\min(M \setminus \tau))$, which is $\sigma_{N',N}(\alpha_M)$. 
As $w$ is a condition, 
$\langle \sigma_{N',N}(\alpha_M), \sigma_{N',N}(\alpha_M) \rangle$ 
is in $x_w$. 
So the image of this pair under $\sigma_{N,N'}$, namely 
$\langle \alpha_M, \alpha_M \rangle$, is in $z$.

\bigskip

Now we show that $z$ consists of nonoverlapping pairs. 
Suppose that $\langle \alpha, \alpha' \rangle$ and 
$\langle \gamma, \gamma' \rangle$ are in $z$, and we will prove that it is 
not the case that $\alpha < \gamma \le \alpha'$. 
If both pairs are in $x_r$ then we are done since $r$ is a condition, 
so assume not. 
Suppose for a contradiction that $\alpha < \gamma \le \alpha'$.

\bigskip

\emph{Case 1:} $\langle \alpha, \alpha' \rangle = 
\sigma_{N,N'}(\langle \alpha_0, \alpha_0' \rangle)$ for some 
$\langle \alpha_0, \alpha_0' \rangle \in x_w$ and $N'$ in $A_r$ 
isomorphic to $N$, 
and $N' \cap [\gamma,\gamma'] = \emptyset$. 
Let $\tau := \min(N' \setminus \gamma)$, which exists because 
$\alpha' \in N'$. 
Then $\alpha < \tau \le \alpha'$. 
Since $s$ satisfies requirement (3), 
$\langle \tau, \tau \rangle$ is in $z \cap N'$. 
Let $\tau_0 := \sigma_{N',N}(\tau)$. 
Since $s$ satisfies requirement (5), 
$\langle \tau_0, \tau_0 \rangle$ is in $z \cap N = x_w$. 
But then $\alpha_0 < \tau_0 \le \alpha_0'$, contradicting 
that $w$ is a condition.

\bigskip

\emph{Case 2:} $\langle \alpha, \alpha' \rangle = 
\sigma_{N,N'}(\langle \alpha_0, \alpha_0' \rangle)$ for some 
$\langle \alpha_0, \alpha_0' \rangle \in x_w$ and $N'$ in $A_r$ 
isomorphic to $N$, 
and $N' \cap [\gamma,\gamma'] \ne \emptyset$. 
Then by requirement (3), $\gamma$ and $\gamma'$ are in $N'$. 
Let $\gamma_0 = \sigma_{N',N}(\gamma)$ and $\gamma_0' = 
\sigma_{N',N}(\gamma')$. 
Since $s$ satisfies requirement (5), 
$\langle \gamma_0, \gamma_0 \rangle$ 
is in $z \cap N = x_w$. 
But then $\alpha_0 < \gamma_0 \le \alpha_0'$, contradicting 
that $w$ is a condition.

\bigskip

\emph{Case 3:} $\langle \gamma, \gamma' \rangle = 
\sigma_{N,N'}(\langle \gamma_0, \gamma_0' \rangle)$ for some 
$\langle \gamma_0, \gamma_0' \rangle \in x_w$ and $N' \in A_r$ 
isomorphic to $N$. 
Since $\alpha < \gamma \le \alpha'$, 
$N' \cap [\alpha,\alpha'] \ne \emptyset$. 
Since $s$ satisfies requirement (3), $\alpha$ and $\alpha'$ are in $N'$. 
Let $\alpha_0 := \sigma_{N',N}(\alpha)$ and 
$\alpha_0' = \sigma_{N',N}(\alpha')$. 
Since $s$ satisfies requirement (5), $\langle \alpha_0, \alpha_0' \rangle$ is in 
$z \cap N = x_w$. 
But then $\alpha_0 < \gamma_0 \le \alpha_0'$, contradicting that 
$w$ is a condition.

\bigskip

Since at least one of $\langle \alpha, \alpha' \rangle$ and 
$\langle \gamma, \gamma' \rangle$ is not in $x_r$, these cases cover 
all possibilities.

\bigskip

(4) Let $\zeta \in R_C$ be given, and we will show that 
$\langle \zeta, \zeta \rangle \in z$. 
By Proposition 5.2, $\zeta$ is in the set 
$$
R_{A_r} \cup \{ \min(K \setminus \xi) : K \in C, \ \xi \in R_{A_r} \} \cup 
\{ \sigma_{N,N'}(\tau) : N' \in A, \ N \cong N', \ \tau \in R_{A_w} \}.
$$
Let $\zeta \in R_C$ be given, and we will show that 
$\langle \zeta, \zeta \rangle \in z$. 

\bigskip

\emph{Case 1:} $\zeta \in R_{A_r}$. 
Since $r$ is a condition, $\langle \zeta, \zeta \rangle \in x_r$.

\bigskip

\emph{Case 2:} $\zeta = \sigma_{N,N'}(\tau)$ for some 
$\tau \in R_{A_w}$ and $N' \in A_r$ isomorphic to $N$. 
Since $w$ is a condition, $\langle \tau, \tau \rangle \in x_w$. 
So $\sigma_{N,N'}(\langle \tau, \tau \rangle) = 
\langle \zeta, \zeta \rangle$ is in $z$ by definition.

\bigskip

\emph{Case 3:} For some $K \in C$ and $\xi \in R_{A_r}$, 
$\zeta = \min(K \setminus \xi)$. 
Then $\langle \xi, \xi \rangle \in x_r$. 
First assume that $K \in A_r$. 
Then since $r$ is a condition, $\langle \zeta, \zeta \rangle \in x_r$.

Secondly assume that $K = \sigma_{N,N'}(L)$ for some $L \in A_w$ and 
$N' \in A_r$ isomorphic to $N$. 
Let $\tau := \min(N' \setminus \xi)$. 
Then $\zeta = \min(K \setminus \tau)$. 
Since $r$ is a condition, $\langle \tau, \tau \rangle \in x_r$. 
Let $\tau_0 := \sigma_{N',N}(\tau)$. 
Then $\langle \tau_0, \tau_0 \rangle \in x_r \cap N = x_w$. 
Since $\zeta = \min(K \setminus \tau)$, applying $\sigma_{N',N}$ we get 
that $\zeta_0 := \sigma_{N',N}(\zeta) = \min(L \setminus \tau_0)$. 
Since $w$ is a condition, $\langle \zeta_0, \zeta_0 \rangle$ is in $x_w$. 
Therefore $\sigma_{N,N'}(\langle \zeta_0, \zeta_0 \rangle) = 
\langle \zeta, \zeta \rangle$ is in $z$.
\end{proof}

\begin{proposition}
The forcing poset $\p$ is an $(S,\mathcal Y)$ coherent adequate type forcing.
\end{proposition}

\begin{proof}
Let $(x,A)$ be a condition in $\p$, and we will show that $(x,A)$ satisfies 
properties (I)--(V) in the definition of an $(S,\mathcal Y)$ coherent adequate type forcing. 

(I) is immediate, and (V) follows from Proposition 7.2. 
(II) By the definition of $\p$, 
$A$ is a finite coherent adequate subset of $\mathcal Y$. 
If $M$ and $N$ are in $A$ and $\zeta \in R_M(N)$, then 
$\langle \zeta, \zeta \rangle \in x$. 
This implies that $\zeta \in S$. 
Therefore $A$ is $(S)$ adequate. 

(III) Let $(y,B) \le (x,A)$, and assume that $N$ and $N'$ 
are isomorphic in $B$ with $(x,A) \in N$. 
Let $\sigma := \sigma_{N,N'}$. 
We will show that $\sigma((x,A))$ is in $\p$ and $(y,B) \le \sigma((x,A))$. 
All the properties of being in $\p$ described in Definition 7.1 
are first order definable in $H(\omega_2)$ without parameters, except 
for the requirements on membership in $S$ and $\mathcal Y$. 
Therefore as $N$ and $N'$ are elementary in $H(\omega_2)$ and $\sigma$ 
preserves membership in $S$ and $\mathcal Y$, $(x,A)$ being in $\p$ 
implies that $\sigma((x,A))$ is in $\p$. 
Also $\sigma((x,A)) = (\sigma[x],\sigma[A])$, and since $(y,B)$ is a condition 
with $x \subseteq y$ and $A \subseteq B$, 
$\sigma[x] \subseteq y$ and $\sigma[A] \subseteq B$. 
So $(y,B) \le \sigma((x,A))$.

(IV) Suppose that $M_0, \ldots, M_n$ are isomorphic sets in 
$\mathcal Y$ such that 
$\{ M_0, \ldots, M_n \}$ is $(S)$ coherent adequate and 
$(x,A) \in M_0 \cap \cdots \cap M_n$. 
Let $C = A \cup \{ M_0, \ldots, M_n \}$. 
By Lemma 1.17, $C$ is coherent adequate.

Let $M$ and $N$ be in $C$. 
If $M$ is in $\{ M_0, \ldots, M_n \}$ and $N$ is in $A$, then  
$N \in M$. 
Hence $R_N(M)$ and $R_M(N)$ are empty. 
It follows that $R_C = R_A \cup \{ R_{M_i}(M_j) : i , j \le n \}$. 
In particular, $R_C \subseteq S$. 
So $C$ is $(S)$ coherent adequate.

Define 
$$
z = y \cup \{ \sigma_{K,L}(a) : K, L \in C, \ K \cong L, \ a \in y \cap K \},
$$
where $y = x \cup \{ \langle \zeta, \zeta \rangle : \zeta \in R_C \setminus R_A \}$. 
By Lemma 4.1, $(z,C)$ is closed. 
We will show that $(z,C)$ is condition. 
Then clearly $(z,C) \le (x,A)$ and $M_0, \ldots, M_k \in C$, which finishes 
the proof.

First we claim that $z$ is equal to 
$$
z' := x \cup \{ \sigma_{M_i,M_j}(\langle \zeta, \zeta \rangle) : 
i, j \le n, \ \zeta \in (R_C \cap M_i) \setminus R_A \}.
$$
Obviously $z' \subseteq z$.

For the other direction, first let us prove that $y \subseteq z'$. 
By definition, $x \subseteq z'$. 
Now consider $\zeta \in R_C \setminus R_A$. 
Then $\zeta \in R_{M_i}(M_j)$ for some $i, j \le n$. 
Hence $\langle \zeta, \zeta \rangle = \sigma_{M_j,M_j}(\langle \zeta, \zeta \rangle)$ is in $z'$.

Secondly assume that $b = \sigma_{K,L}(a)$ where $K, L \in C$, 
$K \cong L$, and $a \in y \cap K$. 
Then either $K, L \in A$ or $K, L \in \{ M_0, \ldots, M_n \}$. 
In the former case, $a$ being in $K$ implies that $a$ is not in 
$R_C \setminus R_A$, since every member of $R_C \setminus R_A$ 
lies above every ordinal in $K \cap \omega_2$. 
Hence $a$ is in $x$. 
So $b \in x$ since $(x,A)$ is a condition. 

Assume that $K, L \in \{ M_0, \ldots, M_n \}$. 
If $a \in x$ then $a \in K \cap L$, so $b = \sigma_{K,L}(a) = a$. 
So $b \in y$, and hence $b \in z'$. 
Otherwise $a = \langle \zeta, \zeta \rangle$ where $\zeta$ is in 
$R_C \setminus R_A$, and then by definition 
$b$ is in $z'$.

We have shown that $z = z'$. 
Note that if $\langle \alpha, \alpha \rangle$ is in $z \setminus x$, then  
for all $K \in A$ and $\langle \gamma, \gamma' \rangle$ in 
$x$, $\sup(K \cap \omega_2)$ and $\gamma'$ are below $\alpha$. 
For suppose $\langle \alpha, \alpha \rangle = 
\sigma_{M_i,M_j}(\langle \zeta, \zeta \rangle)$ for some 
$\zeta \in (R_C \cap M_i) \setminus R_A$. 
Then $\zeta \in R_{M_k}(M_l)$ for some $k, l \le n$. 
So $\beta_{M_k,M_l} \le \zeta$. 
Since $(x,A) \in M_k \cap M_l$, any ordinals in $\omega_2$ 
definable from $(x,A)$ are in $\beta_{M_k,M_l}$ and hence are 
below $\zeta$. 
If $\sigma_{M_i,M_j}(\zeta) = \zeta$ then we are done. 
Otherwise $\beta_{M_i,M_j} \le \zeta$, so clearly 
$\beta_{M_i,M_j} \le \sigma_{M_i,M_j}(\zeta)$. 
And as $(x,A) \in M_i \cap M_j$, again any ordinals definable from 
$(x,A)$ are below $\sigma_{M_i,M_j}(\zeta)$.

We now verify that $(z,C)$ satisfies properties (1)--(5) in the definition of $\p$. 
We already noted that $C$ is a coherent adequate subset of 
$\mathcal Y$ and $(z,C)$ is closed. 
Hence properties (2) and (5) holds. 
Also (4) is immediate.

\bigskip

(1) Since isomorphisms preserve membership in $S$, easily $z$ consists of 
pairs of the form $\langle \alpha, \alpha' \rangle$ where 
$\alpha \le \alpha' < \omega_2$ and $\alpha \in S$. 
Pairs which lie in $x$ are nonoverlapping, and pairs lying in 
$z \setminus x$ are nonoverlapping because the first and second 
components in such a pair are equal. 
Consider $\langle \alpha, \alpha' \rangle \in x$ and 
$\sigma_{M_i,M_j}(\zeta)$ where $i, j \le n$ and 
$\zeta \in (R_C \cap M_i) \setminus R_A$. 
By the comment above, $\alpha' < \sigma_{M_i,M_j}(\zeta)$. 
Hence $\langle \alpha, \alpha' \rangle$ and 
$\sigma_{M_i,M_j}(\langle \zeta, \zeta \rangle)$ do not overlap.

(3) Let $\langle \alpha, \alpha' \rangle \in z$ and $K \in C$ with 
$K \cap \omega_2 \nsubseteq \alpha$. 
If $\langle \alpha, \alpha' \rangle \in x$ and $K \in A$, then we are done. 
If $\langle \alpha, \alpha' \rangle \in x$ and $K \in \{ M_0, \ldots, M_n \}$, 
then $\alpha$ and $\alpha'$ are in $K$. 

Suppose that $\langle \alpha, \alpha' \rangle$ is equal to 
$\sigma_{M_i,M_j}(\langle \zeta, \zeta \rangle)$ for some $i, j \le n$ 
and $\zeta \in (R_C \cap M_i) \setminus R_A$. 
Then $\alpha = \alpha'$. 
If $K \in A$, then by the comment above, 
$\sup(K \cap \omega_2) < \alpha$, contradicting our assumption. 
Therefore $K \in \{ M_0, \ldots, M_n \}$. 
If $\alpha \in K$ then we are done, so assume not.
Then $\alpha \in M_j \setminus K$. 
Since $K \cap \beta_{K,M_j} = M_j \cap \beta_{K,M_j}$, 
$\beta_{K,M_j} \le \alpha$. 
So $\alpha \in M_j \setminus \beta_{K,M_j}$. 
It follows that $\alpha_K = \min(K \setminus \alpha)$ is in $R_{M_j}(K)$.  
By definition of $z$, $\langle \alpha_K, \alpha_K \rangle$ is in $z$.
\end{proof}

It remains to show that $\p$ forces that $\dot C_S$ is a club. 
For unboundedness, given a condition $p$ and an ordinal $\gamma$, 
choose $\beta$ in $S$ 
larger than all ordinals appearing in pairs of $p$ 
and all suprema of models appearing in $p$ intersected 
with $\omega_2$. 
Then $(x_p \cup \{ \langle \beta, \beta \rangle \}, A_p)$ is a condition 
which forces that $\dot C_S$ is not a subset of $\gamma$.

The proof of the closure of $\dot C_S$ is similar to the argument from 
\cite{jk24}, except that we have the new problem of needing to 
close under 
isomorphisms when adding something to a condition. 
This problem is dealt with by the next lemma.

\begin{lemma}
Let $(x,A)$ be a condition and let $\langle \alpha, \alpha' \rangle$ be a pair 
such that $\alpha \le \alpha' < \omega_2$ and $\alpha \in S$. 
Assume:
\begin{enumerate}
\item $\langle \alpha, \alpha' \rangle$ does not overlap any pair in $x$;
\item if $N \in A$ and 
$N \cap \omega_2 \nsubseteq \alpha$, then $N \cap [\alpha,\alpha'] \ne 
\emptyset$ implies that $\alpha$ and $\alpha'$ are in $N$, and 
$N \cap [\alpha,\alpha'] = \emptyset$ implies that 
$\langle \alpha_N, \alpha_N \rangle \in x_p$.
\end{enumerate}
Let $z$ be the set 
$$
x \cup \{ \langle \alpha, \alpha' \rangle \} \cup 
\{ \sigma_{N,N'}(\langle \alpha, \alpha' \rangle) : 
N, N' \in A_r, \ N \cong N', \ \langle \alpha, \alpha' \rangle \in N \}.
$$
Then $(z,A)$ is a condition below $(x,A)$ and 
$\langle \alpha, \alpha' \rangle \in z$.
\end{lemma}

\begin{proof}
We will prove that $(z,A)$ is a condition. 
Then it is clear that $(z,A) \le (x,A)$ and $\langle \alpha, \alpha' \rangle \in x$. 
Properties (2) and (4) in the definition of $\p$ are immediate. 

\bigskip

(5) Let $y := x \cup \{ \langle \alpha, \alpha' \rangle \}$. 
Since $(x,A)$ is closed, it is easy to see that $z$ is equal to 
$$
y \cup \{ \sigma_{N,N'}(a) : N, N' \in A, \ N \cong N', \ a \in y \cap N \}.
$$
By Lemma 4.1, $(z,A)$ is closed.

\bigskip

(3) Consider a pair $\langle \alpha_1, \alpha_1' \rangle = 
\sigma_{N,N'}(\langle \alpha, \alpha' \rangle)$, 
where $N$ and $N'$ are isomorphic in $A$ and 
$\langle \alpha, \alpha' \rangle \in N$. 
Let $M \in A$ be given such that $M \cap \omega_2 \nsubseteq \alpha_1$.

We claim that $R_A \cap (\alpha_1,\alpha_1'] = \emptyset$. 
Suppose for a contradiction that $\zeta$ is in this intersection. 
Then $\langle \zeta, \zeta \rangle \in x$. 
So $\sigma_{N',N}(\langle \zeta, \zeta \rangle) \in x$. 
But this last pair overlaps with $\langle \alpha, \alpha' \rangle$, which 
contradicts our assumptions.

\bigskip

\emph{Case 1:} $M \cap [\alpha_1,\alpha_1'] \ne \emptyset$. 
We will show that $\alpha_1$ and $\alpha_1'$ are in $M$.
Suppose that there is $\theta \in M \cap [\alpha_1,\alpha_1']$ such that 
$\beta_{M,N'} \le \theta$. 
Let $\zeta := \min(N' \setminus \theta)$. 
Then $\zeta$ is in $R_A$ and 
$\alpha_1 < \zeta \le \alpha_1'$, contradicting the claim above. 
Therefore any ordinal in the nonempty intersection 
$M \cap [\alpha_1,\alpha_1']$ is strictly below $\beta_{M,N'}$. 
In particular, $\alpha_1 < \beta_{M,N'}$. 

Assume that $N' \le M$. 
If $\beta_{M,N'} \le \alpha_1'$, then 
$\min(N' \setminus \beta_{M,N'})$ is in 
$R_A \cap (\alpha_1,\alpha_1']$, contradicting the claim above.
So $\alpha_1' < \beta_{M,N'}$.
Then $\alpha_1$ and $\alpha_1'$ are 
in $N' \cap \beta_{M,N'}$ and hence in $M$.
 
Now assume that $M < N'$. 
Fix $N^*$ in $A$ isomorphic to $N'$ such that $M \in N^*$. 
Since $M \in N^*$, $\beta_{M,N'} \le \beta_{N^*,N'}$. 
In particular, $\alpha_1 < \beta_{N^*,N'}$. 
If $\beta_{N^*,N'} \le \alpha_1'$, then 
$\min(N' \setminus \beta_{N^*,N'})$ is in $R_A \cap (\alpha_1,\alpha_1']$, 
contradicting the claim. 
Therefore $\alpha_1' < \beta_{N^*,N'}$. 
It follows that $\alpha_1$ and $\alpha_1'$ are in $N^*$.

By Lemma 1.13, $\langle \alpha, \alpha' \rangle = 
\sigma_{N',N}(\langle \alpha_1, \alpha_1' \rangle)$ is equal to 
$\sigma_{N^*,N}(\langle \alpha_1, \alpha_1' \rangle)$. 
Since $M \cap [\alpha_1,\alpha_1'] \ne \emptyset$, applying 
$\sigma_{N^*,N}$ we get that 
$\sigma_{N^*,N}(M) \cap [\alpha,\alpha'] \ne \emptyset$. 
As $\sigma_{N^*,N}(M) \in A$, our assumptions imply that 
$\alpha$ and $\alpha'$ are in $\sigma_{N^*,N}(M)$. 
Applying $\sigma_{N,N^*}$, we get that 
$\alpha_1$ and $\alpha_1'$ are in $M$.

\bigskip

\emph{Case 2:} $M \cap [\alpha_1,\alpha_1'] = \emptyset$. 
Let $\beta := \min(M \setminus \alpha_1)$, and we will show that 
$\langle \beta, \beta \rangle \in z$. 
Note that $\beta$ equals $\min(M \setminus \alpha_1')$. 
So if $\beta_{M,N'} \le \alpha_1'$, then 
$\beta \in R_A$ and hence $\langle \beta, \beta \rangle \in x$. 
So assume that $\alpha_1' < \beta_{M,N'}$. 
Since $M$ does not contain $\alpha_1$, we must have that $M < N'$. 
If $\beta_{M,N'} \le \beta$, then $\beta = \min(M \setminus \beta_{M,N'})$. 
So $\beta \in R_A$ and again $\langle \beta, \beta \rangle \in x$. 
Assume that $\beta < \beta_{M,N'}$. 
Since $M < N'$, it follows that $\beta \in M \cap N'$. 

Fix $N^*$ in $A$ which is isomorphic to $N'$ with $M \in N^*$. 
Since $\beta \in N' \cap N^*$, $\beta < \beta_{N',N^*}$. 
So $\alpha_1$ and $\alpha_1'$ are in $N^*$. 
By Lemma 1.13, $\sigma_{N^*,N}(\langle \alpha_1,\alpha_1' \rangle) = 
\sigma_{N',N}(\langle \alpha_1, \alpha_1' \rangle) = 
\langle \alpha, \alpha' \rangle$. 
Since $M \cap [\alpha_1,\alpha_1'] = \emptyset$, applying  
$\sigma_{N^*,N}$ we get that 
$\sigma_{N^*,N}(M) \cap [\alpha,\alpha'] = \emptyset$. 
Let $\pi := \min(\sigma_{N^*,N}(M) \setminus \alpha)$. 
As $\sigma_{N^*,N}(M) \in A$, our assumptions imply that 
$\langle \pi, \pi \rangle \in x$. 
But $\beta = \min(M \setminus \alpha_1)$, 
so applying $\sigma_{N^*,N}$ we 
get that $\sigma_{N^*,N}(\beta) = 
\min(\sigma_{N^*,N}(M) \setminus \alpha) = \pi$. 
Hence $\langle \sigma_{N^*,N}(\beta), \sigma_{N^*,N}(\beta) \rangle$ 
is in $x$. 
Since $(x,A)$ is closed, $\langle \beta, \beta \rangle \in x$.

\bigskip

(1) It is immediate that the pairs of $z$ are of the correct form. 
We will show that they are nonoverlapping. 
This is true by assumption for all pairs in 
$x \cup \{ \langle \alpha, \alpha' \rangle \}$. 

\bigskip

\emph{Case 1:} Suppose that $\langle \gamma, \gamma' \rangle \in x$ and 
$\langle \alpha_1, \alpha_1' \rangle = 
\sigma_{N,N'}(\langle \alpha, \alpha' \rangle)$, where 
$N$ and $N'$ are isomorphic in $A$ and 
$\langle \alpha, \alpha' \rangle \in N$. 
First assume that $N' \cap [\gamma,\gamma'] \ne \emptyset$. 
Then $\gamma$ and $\gamma'$ are in $N'$. 
Hence any overlap between $\langle \alpha_1, \alpha_1' \rangle$ 
and $\langle \gamma, \gamma' \rangle$ translates to an overlap between 
$\langle \alpha, \alpha' \rangle$ and 
$\sigma_{N',N}(\langle \gamma, \gamma' \rangle)$. 
As $\sigma_{N',N}(\langle \gamma, \gamma' \rangle) \in x$, 
this would contradict our assumptions.

Assume that $N' \cap [\gamma,\gamma'] = \emptyset$. 
Then the only possible overlap 
between $\langle \gamma, \gamma' \rangle$ and 
$\langle \alpha_1, \alpha_1' \rangle$ would be if 
$\alpha_1 < \gamma \le \gamma' < \alpha_1'$. 
Letting $\zeta := \min(N' \setminus \gamma)$, 
we have that 
$\langle \zeta, \zeta \rangle \in x$ and $\alpha_1 < \zeta \le \alpha_1'$.
But now we get a contradiction exactly as in the previous paragraph.

\bigskip

\emph{Case 2:} Suppose that $\langle \alpha_1, \alpha_1' \rangle = 
\sigma_{N,N'}(\langle \alpha, \alpha' \rangle)$ and 
$\langle \alpha_2, \alpha_2' \rangle = 
\sigma_{M,M'}(\langle \alpha, \alpha' \rangle)$, where $N$, $N'$, $M$, $M'$ 
are in $A$, $N \cong N'$, and $M \cong M'$. 
Suppose for a contradiction that $\alpha_1 < \alpha_2 \le \alpha_1'$.
Then $M' \cap [\alpha_1,\alpha_1'] \ne \emptyset$. 
By property (3) of $(z,C)$, 
this implies that $\alpha_1$ and $\alpha_1'$ are in $M'$. 
Since these ordinals are also in $N'$, 
$\alpha_1' < \beta_{M',N'}$.

First assume that $M$ and $N$ are isomorphic. 
As $\alpha$ and $\alpha'$ are in $M \cap N$, 
$\sigma_{M,N}(\alpha) = \alpha$ and $\sigma_{M,N}(\alpha') = \alpha'$. 
Since $\alpha_1 < \alpha_2 \le \alpha_1' < \beta_{M',N'}$, 
these three ordinals are in $M' \cap N'$, so 
$\sigma_{M',N'}$ fixes them.  
In particular, we have that $\alpha_1 = \sigma_{N,N'}(\alpha) = 
\sigma_{M',N'}(\sigma_{M,M'}(\sigma_{N,M}(\alpha))) = 
\sigma_{M',N'}(\sigma_{M,M'}(\alpha)) = 
\sigma_{M',N'}(\alpha_2) = \alpha_2$. 
But this contradicts that $\alpha_1 < \alpha_2$.

Assume that $M < N$. 
Fix $N^*$ in $A$ isomorphic to $N$ with $M \in N^*$. 
Then $\alpha$ and $\alpha'$ are in $N \cap N^*$. 
So $\sigma_{N,N'}$ and $\sigma_{N^*,N'}$ agree on $\alpha$ and 
$\alpha'$ by Lemma 1.13. 
Let $L := \sigma_{N^*,N'}(M)$, which is in $A$ and satisfies that 
$\sigma_{M,L} = \sigma_{N^*,N'} \restriction M$. 
In particular, $\sigma_{M,L}(\langle \alpha, \alpha' \rangle) = 
\sigma_{N^*,N'}(\langle \alpha, \alpha' \rangle) = 
\sigma_{N,N'}(\langle \alpha, \alpha' \rangle) = 
\langle \alpha_1, \alpha_1' \rangle$. 
So $\langle \alpha_1, \alpha_1' \rangle = 
\sigma_{M,L}(\langle \alpha, \alpha' \rangle)$ and 
$\langle \alpha_2, \alpha_2' \rangle = 
\sigma_{M,M'}(\langle \alpha, \alpha' \rangle)$, and thus 
we have reduced the current situation to the case considered in 
the previous paragraph. 
A symmetric argument works when $N < M$.
\end{proof}

Note the following consequence of the lemma: 
if $(x,A)$ is a condition and $\langle \beta, \beta' \rangle \in x$, 
then there is $(y,A) \le (x,A)$ such that 
$\langle \beta, \beta \rangle \in y$. 
Also the lemma implies that for all $p$, there is $q \le p$ 
such that for all $\langle \beta, \beta' \rangle \in x_q$, 
$\langle \beta, \beta \rangle \in x_q$. 
For the process of adding $\langle \beta, \beta \rangle$ to $p$ described in 
the lemma adds only pairs of the form $\langle \gamma, \gamma \rangle$, 
where for some $\gamma'$, $\langle \gamma, \gamma' \rangle \in x_p$. 
So repeating finitely many times we can close the condition under this 
requirement.

\begin{proposition}
The forcing poset $\p$ forces that $\dot C_S$ is a club.
\end{proposition}

\begin{proof}
Suppose that $p$ forces that $\alpha$ is a limit point of $\dot C_S$. 
We will find a condition below $p$ which forces that $\alpha \in \dot C_S$. 
If $\langle \alpha, \alpha' \rangle \in x_p$ for some $\alpha'$, then $p$ forces 
that $\alpha \in \dot C_S$ and we are done. 
So assume not. 
Then for all $\langle \xi, \xi' \rangle$ in $x_p$, either 
$\xi \le \xi' < \alpha$ or $\alpha < \xi \le \xi'$.

Let $A_0 := \{ K \in A_p : \sup(K \cap \omega_2) < \alpha \}$, 
$A_1 := \{ K \in A_p : \sup(K \cap \alpha) < \alpha, 
\ K \cap \omega_2 \nsubseteq \alpha \}$, and 
$A_2 := \{ K \in A_p : \sup(K \cap \alpha) = \alpha \}$. 
Note that for all $M$ and $N$ in $A_2$, $\alpha$ is a limit point of both $M$ 
and $N$ and therefore $\alpha < \beta_{M,N}$.

\bigskip

\emph{Case 1:} $\cf(\alpha) = \omega_1$.  
Extending $p$ if necessary, we may assume that there is $M \in A_p$ 
such $\alpha \in M$. 
Then $M \cap \alpha$ is bounded below $\alpha$. 
Since $p$ forces that $\alpha$ is a limit point of $\dot C_S$, 
we can fix $q \le p$ such that for some 
$\gamma$ and $\gamma'$, $\sup(M \cap \alpha) < 
\gamma \le \gamma' < \alpha$ and $\langle \gamma, \gamma' \rangle \in x_q$. 
Then $M \cap [\gamma,\gamma'] = \emptyset$ and 
$\min(M \setminus \gamma) = \alpha$. 
So $\langle \alpha, \alpha \rangle$ is in $x_q$ since $q$ 
is a condition.

\bigskip

\emph{Case 2:} $\cf(\alpha) = \omega$. 
Since $p$ forces that $\alpha$ is a limit point of $\dot C_S$, fix $t \le p$ 
satisfying:
\begin{enumerate}
\item[(a)] there is $\gamma$ and $\gamma'$ satisfying that 
$\gamma \le \gamma' < \alpha$, 
$\langle \gamma, \gamma' \rangle \in x_t$, and 
for all $K \in A_0 \cup A_1$, $\sup(K \cap \alpha) < \gamma$;
\item[(b)] $\gamma$ is the largest such ordinal;
\item[(c)] for all pairs $\langle \beta, \beta' \rangle$ in $x_t$, 
$\langle \beta, \beta \rangle$ is in $t$.
\end{enumerate} 
Let $q := (x_t,A_p)$. 
Then $q$ is a condition and $q \le p$. 
If $\langle \alpha, \alpha' \rangle \in x_q$ 
for some $\alpha'$ then we are done, 
so assume not. 
It follows that $\langle \alpha, \alpha \rangle$ does not overlap any pair 
in $x_q$ since $\alpha$ is forced to be a limit point of $\dot C_S$. 
Also by the maximality of $\gamma$, the pair 
$\langle \gamma, \alpha \rangle$ 
does not overlap any pair in $x_q$. 
For all $K \in A_1$, $K \cap [\gamma,\gamma'] = \emptyset$, and 
clearly $\alpha_K = \gamma_K$. 
Hence $\langle \alpha_K, \alpha_K \rangle \in x_q$. 
In particular, for all $K \in A_1$, $\alpha < \alpha_K$. 

\bigskip

\emph{Subcase 2a:} $A_2 = \emptyset$. 
We will use Lemma 7.4 to 
show we can add $\langle \gamma, \alpha \rangle$ to $q$, 
contradicting that $\alpha$ is forced to be a limit point of $\dot C_S$. 
By the choice of $\gamma$, $\gamma \in S$ and 
$\langle \gamma, \alpha \rangle$ does not overlap any pair in $x_q$. 
By the case assumption, if $K \in A_p$ and 
$K \cap \omega_2 \nsubseteq \gamma$, 
then $K \in A_1$. 
By the choice of $\gamma$ and the comments above, 
$K \cap [\gamma,\alpha] = \emptyset$ and 
$\langle \alpha_K, \alpha_K \rangle \in x_q$. 
By Lemma 7.4 there is an extension of $q$ which contains 
$\langle \gamma, \alpha \rangle$.

\bigskip

\emph{Subcase 2b:} $A_2 \ne \emptyset$ and 
there exists $M \in A_2$ such that
$\sup(M \cap \omega_2) = \alpha$. 
We apply Lemma 7.4 to 
show that we can add $\langle \alpha, \alpha \rangle$. 
Note that $\alpha \in S$ since $M \in \mathcal Y$. 
And $\langle \alpha, \alpha \rangle$ does not overlap any pair in 
$x_q$ as noted above. 
Let $N \in A_q$ be given such that 
$N \cap \omega_2 \nsubseteq \alpha$ and $\alpha \notin N$. 
Then $N \notin A_0$. 
If $N \in A_1$ then we already know that 
$\langle \alpha_N, \alpha_N \rangle \in x_q$. 
Suppose that $N \in A_2$. 
Then $\alpha < \beta_{M,N}$ as pointed out above. 
In particular, $\alpha = \sup(M \cap \beta_{M,N})$. 
Since $\alpha \notin N$, it is not the case that $M < N$. 
As $\alpha$ is a limit point of $N \cap \beta_{M,N}$ not in $M$, 
likewise $N$ cannot be below $M$. 
Therefore $M \sim N$. 
Since $\alpha_N \notin M$, $\beta_{M,N} \le \alpha_N$. 
So $\alpha < \beta_{M,N} \le \alpha_N$. 
It follows that 
$\alpha_N = \min(N \setminus \beta_{M,N})$. 
So $\alpha_N \in R_M(N)$. 
Therefore $\langle \alpha_N, \alpha_N \rangle \in x_q$.

\bigskip

\emph{Subcase 2c:} $A_2 \ne \emptyset$ and 
for all $M \in A_2$, $M \cap \omega_2 \nsubseteq \alpha$. 
Let $M$ be a member of $A_2$ satisfying that for all $K \in A_2$, 
(i) $M \cap \omega_1 \le K \cap \omega_1$, and 
(ii) if $M \cap \omega_1 = K \cap \omega_1$ then 
$\alpha_M \le \alpha_K$.

First assume that there is $\langle \beta, \beta' \rangle$ in $x_q$ 
with $\alpha \le \beta \le \alpha_M$. 
Then $\beta_M = \alpha_M$. 
By our assumptions, $\beta$ cannot equal $\alpha$, so 
$\alpha < \beta$. 
By the choice of $q$, $\langle \beta, \beta \rangle \in x_q$. 
Therefore $\langle \alpha_M, \alpha_M \rangle \in x_q$. 
In particular, $\alpha_M \in S \cap M$. 
By the definition of $\mathcal Y$, $\sup(M \cap \alpha_M) = \alpha \in S$.

We apply Lemma 7.4 to show that we can add 
$\langle \alpha, \alpha \rangle$. 
The pair $\langle \alpha, \alpha \rangle$ does not overlap any pair in $x_q$. 
Let $N \in A_q$ be given such that $N \cap \omega_2 \nsubseteq \alpha$ 
and $\alpha \notin N$. 
We will show that $\langle \alpha_N, \alpha_N \rangle \in x_q$. 
Obviously $N \notin A_0$. 
And if $N \in A_1$, then we already know that 
$\langle \alpha_N, \alpha_N \rangle \in x_q$. 
Let $N \in A_2$. 
Then by (i), $M \le N$. 
Recall that $\alpha < \beta_{M,N}$. 
Since $\alpha \notin N$ and $\alpha$ is a limit point of $M \cap \beta_{M,N}$, 
it is not possible that $M < N$. 
So $M \cap \omega_1 = N \cap \omega_1$. 
By (ii), $\alpha_M \le \alpha_N$. 
But then $\alpha < \beta \le \alpha_N$, so $\beta_N = \alpha_N$. 
Since $q$ is a condition, 
$\langle \beta_N, \beta_N \rangle \in x_q$. 
So $\langle \alpha_N, \alpha_N \rangle \in x_q$.

Now assume that there is no such pair $\langle \beta, \beta' \rangle$ in $x_q$. 
In particular, $\langle \alpha_M, \alpha_M \rangle$ is not in $x_q$. 
We apply Lemma 7.4 to show that we can add 
$\langle \gamma, \alpha_M \rangle$ to $q$. 
This is a contradiction, since any extension of $q$ forces that 
$\alpha$ is a limit point of $\dot C_S$. 

Since $\langle \gamma, \gamma \rangle \in x_q$, it follows that 
$\langle \gamma_M, \gamma_M \rangle \in x_q$ since $q$ 
is a condition. 
By the maximality of $\gamma$, $\gamma = \gamma_M$. 
So $\gamma \in M$. 
Also by the maximality of $\gamma$, $\langle \gamma, \alpha_M \rangle$ 
does not overlap any pair in $x_q$ of ordinals below $\alpha$. 
And if $\langle \beta, \beta' \rangle \in x_q$ with $\alpha \le \beta$, 
then $\alpha_M < \beta$ by our case assumption.
Hence $\langle \gamma, \alpha_M \rangle$ does not overlap 
any pair in $x_q$.

Now let $N \in A_q$ be given such that $N \cap \omega_2 \nsubseteq \alpha$. 
We will show that either $\gamma$ and $\alpha_M$ are in $N$, or 
$N \cap [\gamma,\alpha_M] = \emptyset$ and 
$\langle \gamma_N, \gamma_N \rangle \in x_q$. 
Obviously $N \notin A_0$. 
Suppose that $N \in A_1$. 
Then $\langle \alpha_N, \alpha_N \rangle \in x_q$ as noted above. 
And by the choice of $\gamma$, $N \cap \alpha \subseteq \gamma$. 
So $\gamma_N = \alpha_N$. 
Hence $\langle \gamma_N, \gamma_N \rangle \in x_q$. 

Suppose that $N \in A_2$. 
First assume that for all $K \in A_2$, 
$N \cap \omega_1 \le K \cap \omega_1$. 
Then by (i) and (ii), 
$N \cap \omega_1 = M \cap \omega_1$ and $\alpha_M \le \alpha_N$. 
Since $\alpha < \beta_{M,N}$ and $\gamma \in M \cap \alpha$, 
$\gamma \in N$. 
We claim $\alpha_M = \alpha_N$. 
If not, then $\beta_{M,N} \le \alpha_M$. 
But then $\alpha < \beta_{M,N} \le \alpha_M$ implies that 
$\alpha_M = \min(M \setminus \beta_{M,N})$. 
So $\alpha_M \in R_A$, and therefore 
$\langle \alpha_M, \alpha_M \rangle \in x_q$, which contradicts our case assumption.
It follows that $\gamma$ and $\alpha_M$ are both in $N$.

Now assume that $M < N$. 
Recall that $\alpha < \beta_{M,N}$. 
So if $\beta_{M,N} \le \alpha_M$, then 
$\alpha_M = \min(M \setminus \beta_{M,N})$ and hence $\alpha_M \in R_N(M)$. 
This implies that $\langle \alpha_M, \alpha_M \rangle \in x_q$, 
which contradicts our case assumption. 
So $\alpha_M < \beta_{M,N}$. 
Since $M \cap \beta_{M,N} \subseteq N$, 
$\gamma$ and $\alpha_M$ are in $N$.
\end{proof}

\section{Preserving cardinals larger than $\omega_2$}

In this final section we will prove a general amalgamation result 
for models of size $\omega_1$. 
As a consequence, we show that the forcing poset consisting of 
finite coherent adequate sets in $H(\lambda)$ ordered by inclusion is 
$\omega_2$-c.c.\footnote{If $2^{\omega_1} = \omega_2$ and 
$\lambda^{<\lambda} = \lambda$, then 
the chain condition follows immediately from the conjunction of 
Proposition 4.4 of \cite{jk21} and Lemma 2.5 of \cite{mota}. 
A proof of Lemma 2.5 of \cite{mota} can be found in Lemma 3.9 of 
\cite{mota2}.}
Although we did not need this fact for the other results in this paper, 
we record it here for future applications.

\begin{lemma}
Let $M_0$, $M_1$, $K_0$, and $K_1$ be in $\mathcal X$. 
Let $\beta \in \Lambda$. 
Assume that $M_0 \cap \beta = M_1 \cap \beta$, 
$K_0 \cap \beta = K_1 \cap \beta$, $\beta_{K_0,M_0} \le \beta$, 
and $\beta_{K_1,M_1} \le \beta$. 
Then $\beta_{K_0,M_0} = \beta_{K_1,M_1}$.
\end{lemma}

\begin{proof}
Clearly $K_0 \cap \beta = K_1 \cap \beta$ and 
$M_0 \cap \beta = M_1 \cap \beta$ imply that 
$\Lambda_{M_0} \cap (\beta+1) = 
\Lambda_{M_1} \cap (\beta+1)$ and 
$\Lambda_{K_0} \cap (\beta+1) = \Lambda_{K_1} \cap (\beta+1)$. 
Since $\beta_{K_0,M_0}$ and $\beta_{K_1,M_1}$ are both in $\beta+1$, 
$\beta_{K_0,M_0} = \max(\Lambda_{K_0} \cap \Lambda_{M_0}) = 
\max(\Lambda_{K_0} \cap \Lambda_{K_1} \cap (\beta+1)) = 
\max((\Lambda_{K_0} \cap (\beta+1)) \cap (\Lambda_{M_0} \cap (\beta+1))) = 
\max((\Lambda_{K_1} \cap (\beta+1)) \cap (\Lambda_{M_1} \cap (\beta+1))) = 
\max(\Lambda_{K_1} \cap \Lambda_{M_1} \cap (\beta+1)) = 
\max(\Lambda_{K_1} \cap \Lambda_{M_1}) = \beta_{K_1,M_1}$.
\end{proof}

\begin{proposition}
Let $A$ be a finite coherent adequate set. 
Let $\chi \ge \lambda$ be a regular cardinal. 
Assume that $N^*$ is an elementary substructure of 
$(H(\chi),\in)$ of size $\omega_1$ such that 
$\beta^* := N^* \cap \omega_2 \in \Lambda$. 
Let $\beta \in \beta^* \cap \Lambda$.

Assume that there exists a map $M \mapsto M'$ from $A$ into $N^*$ satisfying:
\begin{enumerate}
\item $M \cong M'$, $M \cap \beta^* = M' \cap \beta$, 
and $M \cap N^* \subseteq M'$;
\item $K \in M$ iff $K' \in M'$, and in that case, 
$\sigma_{M,M'}(K) = K'$;
\item if $M \in N^*$ then $M = M'$;
\item $A' := \{ M' : M \in A \}$ is a coherent adequate set.
\end{enumerate}
Then $C := A \cup A'$ is a coherent adequate set. 
Moreover, if we let $r_A = \{ \min(M \setminus \beta^*) : M \in A \}$ and 
$r_{A'} = \{ \min(M' \setminus \beta) : M \in A \}$, then  
$R_C \subseteq R_A \cup R_{A'} \cup r_A \cup r_{A'}$.
\end{proposition}

\begin{proof}
For all $M \in A$, 
since $M \cap \beta^* = M' \cap \beta$, $\sigma_{M,M'} \restriction \beta$ is 
the identity function. 
Also note that $M \cap \beta^* = M' \cap \beta$ obviously implies that 
$M \cap \beta^* \subseteq \beta$ and $M \cap \beta = M' \cap \beta$.

First we will prove that $C$ is adequate and that the remainder points of $C$ 
are as required. 
Since $A$ and $A'$ are each adequate, we only need to compare models 
in $A$ with models in $A'$. 
So let $K$ and $M$ be in $A$ and we will compare $K$ and $M'$. 
Since $M' \in N^*$, $M' \cap \omega_2 \subseteq \beta^*$. 
So $\beta_{K,M'} \le \beta^*$ by Lemma 1.7(1).  
If $\beta < \beta_{K,M'}$ 
then by Lemma 1.2, $K \cap [\beta,\beta_{K,M'})$ is nonempty. 
But then $K \cap [\beta,\beta^*)$ is nonempty, which is false. 
So $\beta_{K,M'} \le \beta$.

We claim that either $\beta^* \le \beta_{K,M}$ or 
$\beta_{K,M} = \beta_{K,M'}$. 
In particular, $\beta_{K,M'} \le \beta_{K,M}$. 
For suppose that $\beta_{K,M} < \beta^*$. 
If $\beta < \beta_{K,M}$, then by Lemma 1.7(2), there is 
$\xi \in K \cap M \cap [\beta,\beta_{K,M})$. 
But then $\xi \in K \cap \beta^* \subseteq \beta$, which is a contradiction. 
So $\beta_{K,M} \le \beta$. 
Now apply Lemma 8.1 letting $K_0 = K_1 = K$, $M_0 = M$, 
and $M_1 = M'$. 
Then obviously $K_0 \cap \beta = K_1 \cap \beta$, 
$M_0 \cap \beta = M_1 \cap \beta$, 
$\beta_{K_0,M_0} = \beta_{K,M} \le \beta$, and 
$\beta_{K_1,M_1} = \beta_{K,M'} \le \beta$. 
So by Lemma 8.1, $\beta_{K,M} = \beta_{K,M'}$.

\bigskip

Now we show that $\{ K, M' \}$ is adequate.

\bigskip

\emph{Case 1:} $K < M$. 
Then $K \cap \beta_{K,M} \in M$. 
Since $\beta_{K,M'} \le \beta_{K,M}$, 
$K \cap \beta_{K,M'} \in M$. 
As $\beta_{K,M'} \le \beta$ and 
$K \cap \beta = K' \cap \beta$, 
$K \cap \beta_{K,M'} = K' \cap \beta_{K,M'}$. 
Since $K' \in N^*$, $K \cap \beta_{K,M'} \in M \cap N^*$. 
But $M \cap N^* \subseteq M'$. 
So $K \cap \beta_{K,M'} \in M'$, and therefore $K < M'$.

\bigskip

\emph{Case 2:} $K \sim M$. 
Since $\beta_{K,M'} \le \beta_{K,M}$, 
$K \cap \beta_{K,M'} = M \cap \beta_{K,M'}$. 
As $\beta_{K,M'} \le \beta$ and $M \cap \beta = M' \cap \beta$, 
$M \cap \beta_{K,M'} = M' \cap \beta_{K,M'}$.
So $K \cap \beta_{K,M'} = M' \cap \beta_{K,M'}$.

\bigskip

\emph{Case 3:} $M < K$. 
Then $M \cap \beta_{K,M} \in K$. 
Since $\beta_{K,M'} \le \beta_{K,M}$, 
$M \cap \beta_{K,M'} \in K$. 
But $\beta_{K,M'} \le \beta$ and $M \cap \beta = M' \cap \beta$. 
So $M \cap \beta_{K,M'} = M' \cap \beta_{K,M'}$. 
Therefore $M' \cap \beta_{K,M'} \in K$. 
Hence $M' < K$.

\bigskip

Now we show that the remainder points are as required.

\bigskip

Let $\zeta \in R_K(M')$ be given. 
First assume that $\zeta \ge \beta$. 
If $\zeta = \min(M' \setminus \beta_{K,M'})$, then 
since $\beta_{K,M'} \le \beta$, $\zeta = \min(M' \setminus \zeta)$. 
So $\zeta \in r_{A'}$ and we are done. 
Otherwise $\zeta = \min(M' \setminus \gamma)$ for some 
$\gamma \in K \setminus \beta_{K,M'}$. 
Since $M' \cap \omega_2 \subseteq \beta^*$, $\zeta < \beta^*$. 
As $K \cap \beta^* \subseteq \beta$, $\gamma < \beta$. 
So again $\zeta = \min(M' \setminus \beta)$ and $\zeta \in r_{A'}$.

Secondly, assume that $\zeta < \beta$. 
Then $\zeta \in M' \cap \beta = M \cap \beta$. 
Since $\beta_{K,M'} \le \zeta$, $\zeta$ is not in $K$. 
Suppose that $M \le K$. 
Since $\zeta \in M \setminus K$, $\beta_{K,M} \le \zeta$. 
As $\zeta < \beta$, $\beta_{K,M} < \beta$. 
Therefore by the claim above, $\beta_{K,M} = \beta_{K,M'}$. 
If $\zeta = \min(M' \setminus \beta_{K,M'})$, then 
clearly $\zeta = \min(M \setminus \beta_{K,M})$ and $\zeta \in R_A$. 
Suppose that $\zeta = \min(M' \setminus \gamma)$ for some 
$\gamma \in K \setminus \beta_{K,M'}$. 
Then $\gamma \in K \setminus \beta_{K,M}$ and 
$\zeta = \min(M \setminus \gamma)$, so $\zeta \in R_A$.

Now assume that $K < M$. 
Then $\zeta = \min(M' \setminus \gamma)$ for some 
$\gamma \in K \setminus \beta_{K,M'}$. 
Since $\zeta < \beta$, $\gamma < \beta$. 
Hence $\gamma$ is not in $M$ since otherwise it would be in $M'$. 
So $\gamma \in K \setminus M$. 
Since $K < M$, this implies that $\beta_{K,M} \le \gamma$. 
So $\beta_{K,M} = \beta_{K,M'}$ by the claim above. 
Therefore $\gamma \in K \setminus \beta_{K,M}$ and 
$\zeta = \min(M \setminus \gamma)$. 
Hence $\zeta$ is in $R_A$.

\bigskip

Now consider $\zeta \in R_{M'}(K)$. 
First assume that $\beta \le \zeta$. 
Since $K \cap \beta^* \subseteq \beta$, $\beta^* \le \zeta$. 
But $\beta_{K,M'}$ and any ordinal in $M' \cap \omega_2$ 
are below $\beta^*$. 
So clearly $\zeta = \min(K \setminus \beta^*)$ and $\zeta \in r_A$. 

Now suppose that $\zeta < \beta$. 
First assume that $K \le M$. 
Since $\zeta \notin M'$ and $M' \cap \beta = M \cap \beta$, 
$\zeta \notin M$. 
So $\zeta \in K \setminus M$. 
Since $K \le M$, this implies that $\beta_{K,M} \le \zeta$. 
Therefore $\beta_{K,M} < \beta$ and hence $\beta_{K,M} = \beta_{K,M'}$ 
by the claim above. 
If $\zeta = \min(K \setminus \beta_{K,M'})$, then 
$\zeta = \min(K \setminus \beta_{K,M})$ and 
therefore $\zeta \in R_A$. 
Otherwise $\zeta = \min(K \setminus \gamma)$ for some 
$\gamma \in M' \setminus \beta_{K,M'}$. 
Since $\zeta < \beta$, $\gamma < \beta$. 
So $\gamma \in M' \cap \beta = M \cap \beta$. 
Hence $\gamma \in M \setminus \beta_{K,M}$, so 
$\zeta \in R_A$.

Next assume that $M < K$. 
Then $\zeta = \min(K \setminus \gamma)$ for some 
$\gamma \in M' \setminus \beta_{K,M'}$. 
Since $\zeta < \beta$, $\gamma \in M' \cap \beta = M \cap \beta$. 
So $\gamma \in M$. 
As $\gamma \in M \setminus K$ and $M < K$, 
$\beta_{K,M} \le \gamma < \beta$. 
By the claim above, $\beta_{K,M} = \beta_{K,M'}$. 
Hence $\gamma \in M \setminus \beta_{K,M}$, so  
$\zeta \in R_A$.

\bigskip

Next we will prove that $C$ is coherent. 
We verify the conditions (1), (2), and (3) of Definition 1.15. 

\bigskip

(1) Let $K$ and $M$ be in $A$ and 
assume that $K \sim M'$. 
Then $K \cap \omega_1 = M' \cap \omega_1 = M \cap \omega_1$, 
so $K$ and $M$ are strongly isomorphic. 
Since $M$ and $M'$ are isomorphic, $K$ and $M'$ are isomorphic. 
To see that $K$ and $M'$ are strongly isomorphic, let 
$a \in K \cap M'$ be given. 
Since $a \in M'$ and $M' \in N^*$, $a \in K \cap N^*$. 
Since $K \cap N^* \subseteq K'$, $a \in K'$. 
So $a \in K \cap K'$, which implies that $\sigma_{K,K'}(a) = a$ 
since $K$ and $K'$ are strongly isomorphic by assumption. 
Also $a \in K' \cap M'$, so $\sigma_{K',M'}(a) = a$ since 
$K'$ and $M'$ are strongly isomorphic. 
So $\sigma_{K,M'}(a) = \sigma_{K',M'}(\sigma_{K,K'}(a)) = 
\sigma_{K',M'}(a) = a$.

\bigskip

(2) Let $K$ and $M$ be in $A$. 
First assume that $K < M'$, and we will find $M^*$ in $C$ 
isomorphic to $M'$ such that $K \in M^*$. 
Since $K \cap \omega_1 < M' \cap \omega_1 = M \cap \omega_1$, $K < M$. 
By the coherence of $A$, there is $M^*$ in $A$ isomorphic to $M$ 
such that $K \in M^*$. 
Then $M^*$ is isomorphic to $M'$ and we are done.

Now assume that $M' < K$, and we will find $K^*$ in $C$ 
isomorphic to $K$ such that $M' \in K^*$. 
We have that $M' \cap \omega_1 < K \cap \omega_1 = K' \cap \omega_1$, 
so $M' < K'$. 
Since $A'$ is coherent, there is $K^*$ in $A'$ isomorphic to $K'$ 
such that $M' \in K^*$. 
Then $K^*$ is isomorphic to $K$.

\bigskip

(3) Let $M$, $N$, and $K$ be in $C$, where $M$ and $N$ are isomorphic 
and $K \in M$. 
We will prove that $\sigma_{M,N}(K) \in C$. 

We claim that if $M$ is in $A$ then so is $K$. 
So assume that  $K = L'$ for some $L \in A$. 
Since $L' \in M \cap N^*$ and $M \cap N^* \subseteq M'$, 
$L' \in M \cap M'$. 
Therefore $\sigma_{M,M'}(L') = L'$. 
But $L' \in M'$ implies that $L \in M$ and $\sigma_{M,M'}(L) = L'$. 
Hence $L = L'$ and therefore $L' = K$ is in $A$.

\bigskip

\emph{Case 1:} $M$ and $N$ are in $A$. 
Then $K \in A$, so we are done since $A$ is coherent. 

\bigskip

\emph{Case 2:} $M$ is in $A$ and $N = P'$ for some $P \in A$. 
Then again $K$ is in $A$. 
Let $Q = \sigma_{M,P}(K)$, which is in $A$ since $A$ is coherent. 
Then $\sigma_{P,P'}(Q) = Q'$ is in $A'$. 
And $\sigma_{M,P'}(K) = \sigma_{P,P'}(\sigma_{M,P}(K)) = 
\sigma_{P,P'}(Q) = Q'$ which is in $A'$ and hence in $C$.

\bigskip

\emph{Case 3:} $M = M_0'$ and $N = N_0'$ 
for some $M_0$ and $N_0$ in $A$. 
If $K = L'$ for some $L \in A$, then we are done since $A'$ is coherent. 
Suppose that $K \in A$. 
Then $K \in M_0'$ and $M_0' \in N^*$, which imply that $K \in N^*$. 
So $K = K'$. 
Hence $K \in A'$ so $\sigma_{M,N}(K) \in A'$ since $A'$ is coherent.
\end{proof}

\begin{thm}
Assume CH. 
Then the forcing poset consisting of 
finite coherent adequate sets ordered by inclusion 
is $\omega_2$-c.c.
\end{thm}

\begin{proof}
Let $\mathcal A$ be a maximal antichain, and we will prove that 
$|\mathcal A| \le \omega_1$. 
Fix a regular cardinal $\chi > \lambda$ such that $\p$ is in $H(\lambda)$. 
By CH, we can fix $N^*$ an elementary substructure of $H(\chi)$ 
of size $\omega_1$ 
such that $Y$, $\p$, and $\mathcal A$ are in $N^*$ and 
$(N^*)^\omega \subseteq N^*$.

Since $N^*$ has size $\omega_1$, we will be done if we can show that 
$\mathcal A \subseteq N^*$. 
Let $\beta^* := N^* \cap \omega_2$. 
Note that $\beta^*$ is in $\Lambda$. 
And by elementarity, $\Lambda \cap \beta^*$ is cofinal in $\beta^*$.

Let $A = \{ M_0, \ldots, M_k \}$ be a condition in $\mathcal A$. 
Fix $\beta \in \Lambda \cap N^*$ large enough so that 
for all $M \in A$, $\sup(M \cap \beta^*) < \beta$. 
Let $R$ be the relation where $R(i,j)$ holds if 
$M_i \in M_j$. 
Let $d$ be the set of $i \le k$ such that $M_i \in N^*$.

For each $i \le k$ let 
$\mathfrak M_i$ denote the structure 
$(M_i,\in,Y \cap M_i)$, and let 
$\overline{\mathfrak M}_i$ denote its transitive collapse. 
Recall that $\sigma_{M_i}$ is the transitive collapsing map 
of $\mathfrak{M}_i$. 
For each pair $\langle i, j \rangle$ in $R$, let 
$J_{i,j} := \sigma_{M_j}(M_i)$. 
For each $i \le k$ let $S_i$ be the relation where 
$S_i(a,b)$ holds if $a \in M_i \cap N^*$ and $\sigma_{M_i}(a) = b$.

Since $R$ and $d$ are finite, they are in $N^*$. 
By CH, $H(\omega_1) \subseteq N^*$. 
So each transitive collapse $\overline{\mathfrak{M}}_i$ and each 
object $J_{i,j}$ is in $N^*$. 
As $N^*$ is countably closed, each set $M_i \cap N^*$ is a member of $N^*$. 
So each relation $S_i$ is in $N^*$.

The objects $M_0, \ldots, M_k$ witness in the 
model $H(\chi)$ that there exist $M_0', \ldots, M_k'$ satisfying:
\begin{enumerate}
\item[(a)] 
$\{ M_0', \ldots, M_k' \}$ is a finite coherent adequate set in $\mathcal A$;
\item [(b)] for $i \le k$, $M_i' \cap \beta = M_i \cap \beta^*$, 
$M_i \cap N^* \subseteq M_i'$, and the transitive collapse of 
the structure $(M_i',\in,Y \cap M_i')$ is equal to $\overline{\mathfrak{M}}_i$;
\item [(c)] $R(i,j)$ iff $M_i' \in M_j'$, and in that case, 
$\sigma_{M_j'}(M_i') = J_{i,j}$;
\item [(d)] for $i \le k$, if $S_i(a,b)$ then $\sigma_{M_i'}(a) = b$;
\item [(e)] for $i \in d$, $M_i = M_i'$.
\end{enumerate}
The parameters mentioned in the above statement are those described in the 
previous paragraph together with $\mathcal A$, $\beta$, $M_i \cap \beta^*$ 
for $i \le k$, and $M_i$ for $i \in d$. 
So the parameters are all members of $N^*$. 
By the elementarity of $N^*$, fix $M_0', \ldots, M_k'$ in $N^*$ 
satisfying the same statement.

Let us show that the map $M_i \mapsto M_i'$ for $i \le k$ 
satisfies the assumptions (1)--(4) of Proposition 8.2. 
By (a), $\{ M_0', \ldots, M_k' \}$ is a coherent adequate set, 
so (4) holds. 
For (3), if $M_i \in N^*$, then $i \in d$, and so by (e), $M_i = M_i'$.

(1) Consider $i \le k$. 
By (b), $M_i \cap \beta^* = M_i' \cap \beta$ and 
$M_i \cap N^* \subseteq M_i'$. 
Also by (b), the structures $(M_i,\in,Y \cap M_i)$ and 
$(M_i',\in,Y \cap M_i')$ have the same transitive collapse 
and hence are isomorphic. 
To see that $M_i \cong M_i'$, let $a \in M_i \cap M_i'$. 
Then $a \in M_i \cap N^*$. 
Let $b := \sigma_{M_i}(a)$. 
Then $S_i(a,b)$ holds. 
By (d), $\sigma_{M_i'}(a) = b$. 
Hence $\sigma_{M_i,M_i'}(a) = 
\sigma_{M_i'}^{-1}(\sigma_{M_i}(a)) = 
\sigma_{M_i'}^{-1}(b) = a$. 

(2) By the definition of $R$ and (c), $M_i \in M_j$ iff $R(i,j)$ 
iff $M_i' \in M_j'$. 
And in that case, $\sigma_{M_j'}(M_i') = J_{i,j}$. 
Now $J_{i,j} = \sigma_{M_j}(M_i)$ by definition. 
So $\sigma_{M_j,M_j'}(M_i) = 
\sigma_{M_j'}^{-1}(\sigma_{M_j}(M_i)) = 
\sigma_{M_j'}^{-1}(J_{i,j}) = M_i'$. 

This completes the verification of the assumptions of Proposition 8.2.  
Let $C := A \cup A'$. 
Then $C$ is a finite coherent adequate set, and obviously 
$C \le A, A'$. 
Hence $A$ and $A'$ are compatible. 
But $A$ and $A'$ are both in $\mathcal A$. 
Since $\mathcal A$ is an antichain, $A = A'$. 
Therefore $A \in N^*$. 
Thus we have shown that $\mathcal A \subseteq N^*$, completing the proof.
\end{proof}

By a somewhat easier argument, 
if CH fails then the forcing poset consisting of finite coherent adequate sets ordered 
by inclusion satisfies the $(2^\omega)^+$-c.c. 
Hence by the material in Section 3, if $2^\omega < \lambda$ then 
this forcing poset collapses $2^\omega$ to 
have size $\omega_1$, forces CH, and preserves all cardinals above $2^\omega$.

\bibliographystyle{plain}
\bibliography{paper25}

\end{document}